\documentclass[11pt,a4paper, reqno]{amsart}
\usepackage{mathrsfs}

\usepackage{amsmath,amsfonts,verbatim}
\usepackage{latexsym}
\usepackage{amssymb,leftidx}
\usepackage{extarrows}
\usepackage{overpic}
\usepackage{color}
\usepackage{epsfig}
\usepackage{subfigure}
\usepackage{tikz}
\usepackage[backref=page]{hyperref}

\hypersetup{urlcolor=blue, citecolor=red}

\usepackage{amssymb}
\usepackage{mathrsfs}
\usepackage{amscd}
\usepackage{bbm}
\usepackage{cite}

\newcommand\R{\mathbb{R}}
\newcommand\C{\mathbb{C}}
\newcommand\Z{\mathbb{Z}}
\newcommand\N{\mathbb{N}}

\def\pa{\partial}

\newcommand\A{\bf A}

\newcommand\LL{\mathcal{L}}

\usepackage{graphicx}
\usepackage{float}

\numberwithin{equation}{section}
\newtheorem{proposition}{Proposition}[section]
\newtheorem{definition}{Definition}[section]
\newtheorem{lemma}{Lemma}[section]
\newtheorem{theorem}{Theorem}[section]
\newtheorem{corollary}{Corollary}[section]
\newtheorem{remark}{Remark}[section]

\begin{document}
\title[Decay and Strichartz estimates]{Decay and Strichartz estimates \\ in critical electromagnetic fields }

\author{Xiaofen Gao}
\address{Department of Mathematics, Beijing Institute of Technology, Beijing 100081}
\email{gaoxiaofen@bit.edu.cn}

\author{Zhiqing Yin}
\address{Department of Mathematics, Beijing Institute of Technology, Beijing 100081}
\email{zhiqingyin@bit.edu.cn}

\author{Junyong Zhang}
\address{Department of Mathematics, Beijing Institute of Technology, Beijing 100081}
\email{zhang\_junyong@bit.edu.cn}

\author{Jiqiang Zheng}
\address{Institute of Applied Physics and Computational Mathematics, Beijing 100088}
\email{zhengjiqiang@gmail.com; zheng\_jiqiang@iapcm.ac.cn}

\begin{abstract}
We study the $L^1\to L^\infty$-decay estimates for dispersive equations in the Aharonov-Bohm magnetic fields,
and further prove Strichartz estimates for the Klein-Gordon equation with critical electromagnetic potentials.
The novel ingredients are the construction of Schwartz kernels of the spectral measure and heat propagator for
the Schr\"odinger operator in Aharonov-Bohm magnetic fields. In particular, we explicitly construct the representation of the spectral measure and resolvent of the Schr\"odinger operator
with Aharonov-Bohm potentials, and show that the heat kernel in critical electromagnetic fields satisfies Gaussian boundedness.
In future papers, this result on the spectral measure will be used to (i) study the uniform resolvent estimates, and (ii) prove the $L^p$-regularity property of wave propagation in the same setting.
\end{abstract}

\maketitle

\begin{center}
 \begin{minipage}{120mm}
 \end{minipage}
 \end{center}


\tableofcontents

\section{Introduction and main results}

\subsection{The setting and motivation}
Let us consider the Schr\"odinger operator $\LL_{{\A},a}$ with critical electromagnetic potentials defined by
\begin{equation}\label{LAa}
\mathcal{L}_{{\A},a}=\Big(i\nabla+\frac{{\A}(\hat{x})}{|x|}\Big)^2+\frac{a(\hat{x})}{|x|^2},
\end{equation}
where $\hat{x}=\tfrac{x}{|x|}\in\mathbb{S}^1$, $a\in W^{1,\infty}(\mathbb{S}^{1},\mathbb{R})$ and ${\A}\in W^{1,\infty}(\mathbb{S}^1;\R^2)$  satisfies the transversality condition
\begin{equation}\label{eq:transversal}
{\A}(\hat{x})\cdot\hat{x}=0,
\qquad
\text{for all }x\in\R^2.
\end{equation}
Assume that
\begin{equation}\label{equ:condassa}
  \|a_-\|_{L^\infty(\mathbb{S}^1)}<\min_{k\in\Z}\{|k-\Phi_{\A}|\}^2,
  \qquad
  \Phi_{\A}\notin\Z,
\end{equation}
where $a_-:=\max\{0,-a\}$ is the negative part of $a$, and $\Phi_{\A}$ is the total flux along the sphere
\begin{equation}\label{equ:defphia1}
  \Phi_{\A}=\frac{1}{2\pi}\int_0^{2\pi} \alpha(\theta)\;d\theta,
\end{equation}
where $\alpha(\theta)$ is defined by \eqref{equ:alpha} below. Under the assumption \eqref{equ:condassa}, by using the Hardy inequality (see \cite{LW}, and \cite[cf. (27)]{FFT})
\begin{equation}\label{equ:ghardyLW}
  \min_{k\in\Z}\{|k-\Phi_{\A}|\}^2\int_{\R^2}\frac{|f|^2}{|x|^2}\;dx\leq \int_{\R^2}|\nabla_{\A}f|^2\;dx,
  \qquad
  \nabla_{{\A}}:=i\nabla+{\A},
\end{equation}
the Hamiltonian $\mathcal{L}_{{\A},a}$ can be extended as a self-adjoint operator on $L^2$, via Friedrichs' Extension Theorem (see e.g. \cite[Thm. VI.2.1]{K} and \cite[X.3]{RS}), on the natural form domain
$$
\mathcal D(\mathcal L_{{\A},a})\simeq H^1_{{\A},a}:=\left\{f\in L^2(\R^2;\C):\int_{\R^2}\left(\left|\nabla_{{\A}} f\right|^2+\frac{|f(x)|^2}{|x|^2}|a(\hat x)|\right)\,dx<+\infty\right\}.
$$

The Schr\"odinger operators with electromagnetic potentials have been extensively studied, see Reed-Simon \cite{RS}.
In this paper, in sprit of \cite{Z,CYZ,FZZ}, we study the decay and Strichartz estimates associated with the Schr\"odinger operator \eqref{LAa} in which both the electric and magnetic potentials are singular at origin and scaling critical. \vspace{0.2cm}

It is well-known that the decay estimates and Strichartz estimates are powerful tools to study the dispersive equations.
In this direction, there are many literatures studying the decay behavior of dispersive equations with perturbation of potentials.
Even with subcritical magnetic potentials, there were a sequel of papers (see \cite{CS,DF, DFVV, EGS1, EGS2, S} and the references therein) in which time-decay or Strichartz estimates are studied. For scaling critical purely electric potential,
 the pioneer results are due to Burq, Planchon, Stalker, and Tahvildar-Zadeh  \cite{BPSS, BPST}, in which they proved the validity of Strichartz estimates for the Schr\"odinger and wave equations, in space dimension $n\geq2$. Later, in \cite{FFFP, FFFP1}, Fanelli, Felli, Fontelos, and Primo studied the validity of the time-decay estimate for the Schr\"odinger equation associated with the operator  \eqref{LAa}. For examples, in \cite{FFFP}, they proved the time-decay estimate
\begin{equation}\label{eq:decayshro}
\|e^{it\LL_{{\A},a}}\|_{L^1(\R^2)\to L^\infty(\R^2)}\lesssim |t|^{-1}
\end{equation}
for the Schr\"odinger equation provided \eqref{eq:transversal} and \eqref{equ:condassa} hold.  The Strichartz estimates for $e^{it\LL_{{\A},a}}$ are  consequences of \eqref{eq:decayshro} and the usual Keel-Tao argument \cite{KT}. It is known that ${\A}\sim \/|x|$ is critical for the validity of Strichartz estimates, as proved e.g. in \cite{FG} in the case of the Schr\"odinger equation.
However, the argument in \cite{FFFP, FFFP1} breaks down for wave equation due to the lack of  pseudoconformal invariance (which was used for Schr\"odinger equation). Very recently, Fanelli and the last two authors \cite{FZZ}
established the Strichartz estimate for wave equation by constructing the propagator $\sin(t\sqrt{\mathcal L_{{\A},0}})/\sqrt{\mathcal L_{{\A},0}}$ (based on \emph{Lipschitz-Hankel integral formula}) and showing the local smoothing estimates. However, the method in \cite{FFFP, FFFP1,FZZ} can not be applied to neither half wave propagator $e^{it\sqrt{\LL_{{\A},0}}}$ nor Klein-Gordon evolution. In particular, in the survey \cite{Fanelli}, Fanelli raised open problems about the dispersive estimate for other equations.  The purpose of this paper is to develop the spectral measure to further study
the time decay and Strichartz estimates for dispersive equations. \vspace{0.2cm}

\subsection{Main results}Now we state our main results.  To state our results, let us introduce some preliminary notations. In the following, the Sobolev spaces will be denoted by
\begin{align}\label{def:sobolev}
& \dot H^{s}_{{\A},a}(\R^2):=\mathcal L_{{\A},a}^{-\frac s2}L^2(\R^2),
\qquad
\dot H^s(\R^2):=\dot H^s_{0,0}(\R^2),
\\
&
H^s_{{\A},a}(\R^2) :=L^2(\R^2)\cap\dot H^{s}_{{\A},a}(\R^2),
\qquad
H^s(\R^2):= H^s_{0,0}(\R^2).
\nonumber
\end{align}
We stress that $H_{{\A},a}^1(\R^2)\subset H^1(\R^2)$, and the inclusion is strict, because of the non-integrable singularities of the potentials (see \cite[cf. Lemma 23 - (ii)]{FFT} for details).
Analogously, we define the distorted Besov spaces as follows.  Let $\phi\in C_c^\infty(\mathbb{R}\setminus\{0\})$, with $0\leq\phi\leq 1$, $\text{supp}\,\phi\subset[1/2,1]$, and
\begin{equation}\label{dp}
\sum_{j\in\Z}\phi(2^{-j}\lambda)=1,\quad \phi_j(\lambda):=\phi(2^{-j}\lambda), \, j\in\Z,\quad \varphi_0(\lambda):=\sum_{j\leq0}\phi(2^{-j}\lambda).
\end{equation}

\begin{definition}[Magnetic Besov spaces associated with $\mathcal{L}_{{\A},0}$] For $s\in\R$ and $1\leq p,r<\infty$, the norms of $\|\cdot\|_{\dot{\mathcal{B}}^s_{p,r,\A}(\R^2)}$ and
$\|\cdot\|_{\mathcal{B}^s_{p,r,\A}(\R^2)}$ are given respectively by
\begin{equation}\label{Besov}
\|f\|_{\dot{\mathcal{B}}^s_{p,r,\A}(\R^2)}=\Big(\sum_{j\in\Z}2^{jsr}\|\phi_j(\sqrt{\LL_{{\A},0}})f\|_{L^p(\R^2)}^r\Big)^{1/r},
\end{equation}
and
\begin{equation}\label{Besov'}
\|f\|_{\mathcal{B}^s_{p,r,\A}(\R^2)}=\Big(\|\varphi_0(\sqrt{\LL_{{\A},0}})f\|_{L^p(\R^2)}^r+\sum_{j=1}^\infty 2^{jsr}\|\phi_j(\sqrt{\LL_{{\A},0}})f\|_{L^p(\R^2)}^r\Big)^{1/r}.
\end{equation}
In particular, for $p=r=2$, we have
\begin{equation}\label{Sobolev}
\|f\|_{H^s_{{\A},0}(\R^2)}:=\big\|(1+\mathcal L_{{\A},0})^{\frac s2}f\big\|_{L^2(\R^2)}=\|f\|_{\mathcal{B}^s_{2,2,\A}(\R^2)}.
\end{equation}
\end{definition}
\noindent
The first result about the time decay property for the propagators is the following theorem:
\begin{theorem}\label{thm:dispersive} Let  $\LL_{{\A},0}$ be  in \eqref{LAa} with  $a\equiv0$ where ${\A}\in W^{1,\infty}(\mathbb{S}^{1},\mathbb{R}^2)$ satisfying \eqref{eq:transversal}
such that $\Phi_{\A}\notin\Z$ given in \eqref{equ:defphia1}.
Then there exists a constant $C>0$ such that:

$\bullet$ for Schr\"odinger flow,
\begin{equation}\label{dis-S}
\|e^{it \mathcal L_{{\A},0}}f\|_{L^\infty(\R^2)}\leq C |t|^{-1}\|f\|_{L^1(\R^2)};
\end{equation}

$\bullet$ for half-wave flow,
\begin{equation}\label{dis-w}
\|e^{it\sqrt{\mathcal L_{{\A},0}}}f\|_{L^\infty(\R^2)}\leq C |t|^{-1/2}\|f\|_{\dot{\mathcal{B}}^{3/2}_{1,1,{\A}}(\R^2)};
\end{equation}

$\bullet$ for Klein-Gordon flow,
\begin{equation}\label{dis-kg}
\Big\|e^{it\sqrt{1+\mathcal L_{{\A},0}}}f\Big\|_{L^\infty(\R^2)}\leq C |t|^{-1/2}\|f\|_{\mathcal{B}^{1/2}_{1,1,{\A}}(\R^2)}.
\end{equation}
\end{theorem}
\begin{remark}
The estimate \eqref{dis-S} for the Schr\"odinger equation was proved in \cite{FFFP}, we provide a new proof when electric potential vanishes.
The estimates \eqref{dis-w} and \eqref{dis-kg} are new, at our knowledges.
 \end{remark}

\begin{remark} Inspired by \cite{BFM, CT1},  Fanelli and the last two authors in \cite{FZZ} proved
\begin{equation}\label{dis-w-sin}
\Big\|\frac{\sin(t\sqrt{\mathcal L_{{\A},0}})}{\sqrt{\mathcal L_{{\A},0}}}f\Big\|_{L^\infty(\R^2)}\leq C |t|^{-1/2}\|f\|_{\dot{\mathcal{B}}^{1/2}_{1,1,{\A}}(\R^2)}.
\end{equation}
However, the argument there heavily depends on Lipschitz-Hankel integral formula which does not work for $e^{it\sqrt{\mathcal L_{{\A},0}}}$.
For this reason, we need to use a new strategy developing spectral measure to prove Theorem \ref{thm:dispersive}.
\end{remark}

 \begin{remark}
The assumption $\Phi_{\A}\notin\Z$ is natural. Indeed, if $\Phi_{\A}\in\Z$, by unitary equivalence, the model is same as ${\A}\equiv0$. Hence when $a\equiv0$, the model has
no potential which is trivial.
 \end{remark}

 \begin{remark}
A typical example of ${\bf A}$ is the {\it Aharonov-Bohm} potential
\begin{equation}\label{ab-potential}
{\A}(\hat{x})=\alpha\Big(-\frac{x_2}{|x|},\frac{x_1}{|x|}\Big),\quad \alpha\in\R,
\end{equation}
introduced in the physical literature \cite{AB59}, see also \cite{PT89}.

\end{remark}

\noindent
Using the above dispersive estimates, we can prove Strichartz estimates. Since the Strichartz estimates for Schr\"odinger and wave have been proved in \cite{FFFP,FZZ}, we focus on the Klein-Gordon equations for which the methods in  \cite{FFFP,FZZ} can not be applied.
\vspace{0.2cm}

 In the flat Euclidean space, the free Klein-Gordon equation reads
 \begin{equation}\label{equ:E-KG}
\begin{cases}
\partial_{t}^2u-\Delta u+m^2u=0, \quad (t,x)\in I\times\R^2; \\ u(0)=f(x),
~\partial_tu(0)=g(x).
\end{cases}
\end{equation}
It is well known by \cite{B, GV, KT} that there exists a constant $C>0$ such that
\begin{equation*}
\begin{split}
&\|u(t,x)\|_{L^q_t(I;L^p_x(\R^2))}\leq C \big(\|f\|_{ H^s(\R^2)}+\|g\|_{
H^{s-1}(\R^2)}\big),
\end{split}
\end{equation*}
where $H^s(\R^2)$ is the usual Sobolev space, and the pairs $(q,p)\in \Lambda_{s,\eta}$ with $0\leq\eta\leq1$ (the set $\Lambda_{s,\eta}$ is given in Definition \ref{ad-pair} below).

 \vspace{0.2cm}

\begin{definition}\label{ad-pair}
For $0\leq\eta\leq1$,  we say that a couple $(q,p)\in [2,\infty]\times [2,\infty)$ is admissible,   if $(q,p)$ satisfies
\begin{equation}\label{adm}
\frac{2}q+\frac{1+\eta}p\leq\frac{1+\eta}2.
\end{equation}
For $s\in\R$, we denote $(q,p)\in \Lambda_{s,\eta}$ if $(q,p)$ is admissible and satisfies
\begin{equation}\label{scaling}
\frac1q+\frac {2+\eta}p=\frac {2+\eta}2-s.
\end{equation}

\end{definition}

Throughout this paper, pairs of conjugate indices will be written as $p, p'$, meaning that $\frac{1}p+\frac1{p'}=1$ with $1\leq p\leq\infty$. \vspace{0.2cm}

\begin{theorem}[Global-in-time Strichartz estimate]\label{thm:Stri} Let  $\LL_{{\A},a}$ be as in \eqref{LAa} where $a\in W^{1,\infty}(\mathbb{S}^{1},\mathbb{R})$  and ${\A}\in W^{1,\infty}(\mathbb{S}^{1},\mathbb{R}^2)$ satisfy \eqref{eq:transversal}
and  \eqref{equ:condassa}.  Suppose that $u$ is the
solution to the Cauchy problem
\begin{equation}\label{equ:KG}
\begin{cases}
\partial_{t}^2u+\LL_{{\A},a} u+u=F(t,x), \quad (t,x)\in I\times \R^2; \\ u(0)=u_0(x),
~\partial_tu(0)=u_1(x),
\end{cases}
\end{equation}
for some initial data $(u_0,u_1)\in  H_{{\A},a}^{s}(\R^2)\times H_{{\A},a}^{s-1}(\R^2)$, and
the time interval $I\subseteq\R$, then
\begin{equation}\label{stri}
\begin{split}
&\|u(t,x)\|_{L^q_t(I;L^p_x(\R^2))}+\|u(t,x)\|_{C(I; H_{{\A},a}^s(\R^2))}\\
&\qquad\lesssim \|u_0\|_{ H_{{\A},a}^s(\R^2)}+\|u_1\|_{
H_{{\A},a}^{s-1}(\R^2)}+\|F\|_{L^{\tilde{q}'}_t(I;L^{\tilde{p}'}_x(\R^2))},
\end{split}
\end{equation}
where the pairs $(q,p), (\tilde{q},\tilde{p})\in \Lambda_{s,\eta}$
 with $0\leq\eta\leq1$ and $0\leq s<1$. Furthermore, if either $\Phi_{\A}\not\in\frac12\Z$ or $a(\theta)=a(\cos\theta,\sin\theta)$ is symmetric at $\theta=\pi$, then \eqref{stri} still holds
 for $s\in[1,\frac{2+\eta}2)$.

\end{theorem}

\begin{remark} This is a generalization of Strichartz estimates for wave and Schr\"odinger equations proved in \cite{FFFP, FZZ}. The admissible pairs in $\Lambda_{s,\eta}$ match the wave ones when $\eta=0$ and the Schr\"odinger ones when $\eta=1$.
Since $(q,p)\in \Lambda_{s,\eta}$, if $\Lambda_{s,\eta}$ is not empty, then one has
\begin{equation}\label{s:index}
\frac{2+\eta}2>s=(2+\eta)\Big(\frac12-\frac1p\Big)-\frac1q\geq\frac12(3+\eta)\Big(\frac12-\frac1p\Big)\geq0,
\end{equation}
which is wider than $1>s\geq0$ corresponding to wave \cite{FZZ}. To recover $s\in[1,\frac{2+\eta}2)$, we show the Sobolev embedding inequality
\begin{equation}\label{sob}
\|f\|_{L^q(\R^2)}\leq C\|\LL^{\frac \sigma 2}_{{\A},a}f\|_{L^p(\R^2)},
\end{equation}
where $\sigma=2(\frac1p-\frac1q)\geq0$ which generalizes \cite[Lemma 2.4]{FZZ} to $p\neq 2$.
The Sobolev inequality \eqref{sob} will be proved by establishing the Gaussian boundedness of the heat kernel $e^{-t\LL_{{\A},a}}$
under the assumption that either $\Phi_{\A}\not\in\frac12\Z$ or $a(\theta)$ is symmetric at $\theta=\pi$, i.e. $a(\pi-\theta)=a(\pi+\theta)$ for $\theta\in[0,\pi]$.
In particular, the assumption is satisfied when $a(\theta)\equiv c$ is constant or $a(\theta)=a(\cos\theta)$ which is known as a zonal potential \cite{Gur}
and ``quantum spherical pendulum" in the terminology of \cite{CD}.

\end{remark}

\begin{remark} Due to the magnetic effective, from \eqref{equ:condassa}, it allows negative electric potential in dimension $n=2$.
The result for the 2D model here is different from the result for purely inverse-square electric potential proved in \cite{BPSS, BPST, LSS}.
 In \cite{BPSS},  Burq, Planchon, Stalker, and Tahvildar-Zadeh proved the Strichartz estimates for wave and Schr\"odinger with the purely electric potential
$a|x|^{-2}$ where the constant $a>-(n-2)^2/4$ and $n\geq2$; later in  \cite{BPST}, they generalized the potential to $a(\hat{x})|x|^{-2}$ but with $n\geq3$.
In \cite{LSS},  Lee, Seo and Seok proved the Strichartz estimates (missing endpoint $q=2$) for Klein-Gordon equation with ``small" Fefferman-Phong potentials
in dimension $n\geq3$ which covers the inverse-square potential $a|x|^{-2}$
with small enough constant $|a|$.

\end{remark}

\noindent
\begin{remark}\label{rem:loccount}
The proof of Theorem \ref{thm:Stri} is inspired by the perturbation arguments in \cite{BPST, BPSS}. Nevertheless, to treat $\mathcal L_{{\A},a}$ as a perturbation of $-\Delta$, as in that case, involving local smoothing estimates, would require an estimate like
$$
\big\| |x|^{-\frac12}e^{it\sqrt{1+\LL_{{\A},a}}}f\big\|_{L_{t,x}^2(\R\times\R^2)}\leq \|f\|_{L^2(\R^2)},
$$
to handle the long range perturbation from the magnetic potential. Unfortunately, this estimate is known to be false, even in the free case, by the standard Agmon-H\"ormander Theory.
One can only reach the weight $|x|^{-1/2-\varepsilon}$, see \cite{CF, CF1}.
To overcome this difficulty, we treat $\mathcal L_{{\A},a}$ as an electric perturbation of the purely magnetic operator $\mathcal L_{{\A},0}$ which
leads us to study the spectral measure of $\mathcal L_{{\A},0}$.

\end{remark}

\vspace{0.2cm}

\subsection{Strategy of the proof} Our argument here extend the dispersive and Strichartz estimates of Schr\"odinger \cite{FFFP} and wave \cite{FZZ} to Klein-Gordon equation by constructing the Schwartz kernels of spectral measure. In those papers, the Schr\"odinger propagator $e^{it\LL_{{\A},a}}$ was constructed by using \emph{pseudo-conformal invariance}
and the propagator $\sin(t\sqrt{\mathcal L_{{\A},0}})/\sqrt{\mathcal L_{{\A},0}}$ was showed by using \emph{Lipschitz-Hankel integral formula}, respectively. The current paper,
in the spirit of \cite{HZ, Zhang, ZZ1, ZZ2}, is to express
the propagators by using the spectral measure $d E_{\sqrt{\LL_{{\A},0}}}(\lambda)$ and developing the decay and oscillation properties of the spectral measure.

After expressing the propagator in terms of an integral of the multiplier $e^{it\sqrt{1+\lambda^2}}$ against the spectral measure, our strategy is to use the stationary phase.
 It is known that the Klein-Gordon likes wave at high frequency and likes Schr\"odinger at low frequency, then we combine
\cite{HZ} and \cite{Zhang}, as did in \cite{ZZ3},  to show the localized dispersive estimates. Hence the argument of abstract Strichartz estimate proved in \cite{KT} implies
the Strichartz estimates for Klein-Gordon with purely magnetic potential.  Following \cite{BPST}, we use the local smoothing estimate proved in \cite{CYZ} to treat the perturbation of
the inverse-square potential in Theorem \ref{thm:Stri}. However, it is not enough to cover $(q,p)\in \Lambda_{s,\eta}$ when $s\in [1, (2+\eta)/2)$. To recover this issue,
we prove the heat kernel estimate and Sobolev embedding associated with the operator $\LL_{{\A},a}$. It has been seen the spectral measure and heat kernel are
 two key points in the proof, we give more details about them below.\vspace{0.2cm}

It is known that the Schwartz kernels of the resolvent and spectral measure for Schr\"odinger operator (on manifolds or with potentials)  are the cornerstones of harmonic analysis problems which we plan to investigate in future.
For example, the Schwartz kernels of the resolvent and spectral measure associated with the Schr\"odinger operators in conical singular spaces
have been systematically studied in Hassell-Vasy \cite{HV1,HV2} and Guillarmou-Hassell-Sikora \cite{GHS1, GHS2}. And then the kernels were used to
 study the resolvent estimate in
Guillarmou-Hassell \cite{GH} and
the Strichartz
estimates in Hassell-Zhang \cite{HZ}. Due to the generality and complexity of the geometry, the argument heavily depends on the powerful microlocal strategy developed by Melrose \cite{Melrose1,Melrose2,Melrose3}, and the method is not standard at least for non-micrololcal readers.\vspace{0.2cm}

Very recently, in \cite{Z}, the third author  explicitly constructed the kernels of the resolvent and spectral measure on the flat metric cone, a simple conical singular space.
In contrast to the Laplacian on flat cone, the Schr\"odinger operator \eqref{LAa} here is perturbed by magnetic Aharonov-Bohm potentials, and however has same conical singularity.
Inspired by \cite{Z}, we construct the kernel of the spectral measure of operator \eqref{LAa} when $a=0$. The result about the spectral measure is the following.

\begin{proposition}[Spectral measure kernel]\label{prop:spect}
Let the operator $\LL_{{\A},0}$ be as in Theorem \ref{thm:dispersive}, and let $x=r_1(\cos\theta_1, \sin\theta_1)$ and $y=r_2(\cos\theta_2, \sin\theta_2)$ in $\R^2\setminus\{0\}$. Define
\begin{equation}\label{d-j}
d(r_1,r_2,\theta_1,\theta_2)=\sqrt{r_1^2+r_2^2-2r_1r_2\cos(\theta_1-\theta_2)}=|x-y|,
\end{equation}
and
\begin{equation} \label{d-s}
d_s(r_1,r_2,\theta_1,\theta_2)=\sqrt{r_1^2+r_2^2+2  r_1r_2\, \cosh s},\quad s\in [0,+\infty).
\end{equation}
Then the Schwartz kernel of the spectral measure  $dE_{\sqrt{\mathcal L_{{\A},0}}}(\lambda; x,y)$  can be written as
the sum of the  geometry term
\begin{equation}\label{spect-g}
\begin{split}
\frac{\lambda}{\pi} \sum_{\pm}\Big(a_\pm(\lambda d(r_1,r_2,\theta_1,\theta_2))e^{\pm i\lambda d(r_1,r_2,\theta_1,\theta_2)} A_{\alpha}(\theta_1,\theta_2)\Big)
\end{split}
\end{equation}
and the diffractive term
\begin{equation}\label{spect-d}
\begin{split}
\frac{\lambda}{\pi} \sum_\pm \int_0^\infty a_\pm(\lambda d_s(r_1,r_2,\theta_1,\theta_2))e^{\pm i\lambda d_s(r_1,r_2,\theta_1,\theta_2)}
B_{\alpha}(s,\theta_1,\theta_2) ds,
\end{split}
\end{equation}
where
\begin{equation}\label{A-al}
A_{\alpha}(\theta_1,\theta_2)= \frac{e^{i\int_{\theta_1}^{\theta_2}\alpha(\theta')d\theta'}}{4\pi^2}\big(\mathbbm{1}_{[0,\pi]}(|\theta_1-\theta_2|)
  +e^{-i2\pi\alpha}\mathbbm{1}_{[\pi,2\pi]}(|\theta_1-\theta_2|)\big)
  \end{equation}
 and
\begin{equation}\label{B-al}
\begin{split}
&B_{\alpha}(s,\theta_1,\theta_2)= -\frac{1}{4\pi^2}e^{-i\alpha(\theta_1-\theta_2)+i\int_{\theta_2}^{\theta_{1}} \alpha(\theta') d\theta'}  \Big(\sin(|\alpha|\pi)e^{-|\alpha|s}\\
&\qquad +\sin(\alpha\pi)\frac{(e^{-s}-\cos(\theta_1-\theta_2+\pi))\sinh(\alpha s)-i\sin(\theta_1-\theta_2+\pi)\cosh(\alpha s)}{\cosh(s)-\cos(\theta_1-\theta_2+\pi)}\Big)
 \end{split}
\end{equation}
 and $a_\pm\in C^\infty([0,+\infty))$ satisfies
\begin{equation}\label{bean}
\begin{split}
| \partial_r^k a_\pm(r)|\leq C_k(1+r)^{-\frac{1}2-k},\quad k\geq 0.
\end{split}
\end{equation}
\end{proposition}

\begin{remark}
The decay and oscillation properties of the representation of the spectral measure are the key points to prove the dispersive estimates in Theorem \ref{thm:dispersive}.
\end{remark}

The heat kernel estimates for Schr\"odinger operators on manifolds or with potentials have been extensively studied.
There are too many works to mention all of them here, however
we refer to \cite{MOR,Ko, LS, MS, IKO} and references therein which are closed to our model.

\begin{proposition}[Heat kernel estimate]\label{prop:heat} Let the operator $\LL_{{\A},a}$ be as in Theorem \ref{thm:Stri}.
Assume that either $\Phi_{\A}\not\in\frac12\Z$ or  $a(\theta)=a(\hat{x})$ is symmetric at $\theta=\pi$, i.e. $a(\pi-\theta)=a(\pi+\theta)$ for $\theta\in[0,\pi]$.
Then
there exist constants $c$ and $C$ such that
\begin{align}\label{est:heat}
 | e^{-t\LL_{{\A},a}}(x,y)|\leq C t^{-1}e^{-\frac{|x-y|^2}{ct}}.
\end{align}

\end{proposition}

\begin{remark} In particular $a\equiv0$, we have explicitly constructed the kernel of $e^{-t\LL_{{\A},0}}(x,y)$ in \cite{FZZ}.
Hence we proved
\begin{align}\label{equ:pcontr}
 |e^{-t\LL_{{\A},0}}(x,y)|\lesssim & t^{-1}e^{-\frac{|x-y|^2}{4t}}\quad \forall\;t>0,
\end{align}
and
\begin{align}\label{equ:pcontrder}
 \big|\tfrac{\pa}{\pa t}e^{-t\LL_{{\A},0}}(x,y)\big|\lesssim & t^{-2}e^{-\frac{|x-y|^2}{8t}}\quad \forall\;t>0.
\end{align}
The inequality \eqref{equ:pcontr} first was
proved in \cite{MOR}  is non-trivial since the components of the Aharonov-Bohm potential do not belong to $L^2_{loc}(\R^2)$.
\end{remark}

\begin{remark}
As mentioned above, typical examples potential $a(\hat{x})$ satisfying the assumption are  $a(\hat{x})\equiv c$ and $a(\hat{x})=a(\cos\theta)$ which is known as a zonal potential \cite{Gur}
and ``quantum spherical pendulum" in the terminology of \cite{CD}.
\end{remark}

\begin{remark} Following  \cite{Ko}, one can prove the boundedness of $t^{-1}(1+\frac{t}{|x||y|})^{-\sqrt{\mu_1}}$ but without Gaussian decay term, where $\mu_1$ is the first eigenvalue of $L_{A,a}$.
\end{remark}

\begin{remark} When ${\A}\equiv0$ and $a(\hat{x})\equiv c$ where the constant $c\geq-(n-2)^2/4$ and $n\geq3$,
Liskevich-Sobol \cite{LS} and Milman-Semenov \cite{MS} obtained
that there exist positive constants $C_1,C_2$ and $c_1,c_2$ such that for all
$t>0$ and all $x,y\in \R^n\setminus\{0\}$,
\begin{equation*}
C_1\big( 1\vee \tfrac{\sqrt t}{|x|} \big)^\sigma \big( 1\vee \tfrac{\sqrt t}{|y|} \big)^\sigma t^{-\frac
n2}e^{-\frac{|x-y|^2}{c_1t}}\leq e^{-t\mathcal L_{{0},c}}(x,y)\leq C_2\bigl( 1\vee \tfrac{\sqrt t}{|x|} \bigr)^\sigma \bigl( 1\vee \tfrac{\sqrt t}{|y|} \bigr)^\sigma t^{-\frac
n2}e^{-\frac{|x-y|^2}{c_2t}}.
\end{equation*}
where  $\sigma=\tfrac{n-2}2- \tfrac12\sqrt{(n-2)^2+4c}$ and $A\vee B=\max\{A,B\}$. Recently, Ishige,  Kabeya  and Ouhabaz \cite{IKO}
generalized the result to radial potential $V(|x|)$ which covers $c|x|^{-2}$ but not $a(\hat{x})|x|^{-2}$.

\end{remark}

{\bf Acknowledgments:}\quad  The authors thank Luca Fanelli for helpful discussions. This work  was supported by National Natural
Science Foundation of China (11771041, 11901041, 11831004, 11671033).
\vspace{0.2cm}

\section{Construction of the Spectral measure }

In this section, we prove Proposition \ref{prop:spect} by constructing the Schwartz kernel of spectral measure associated with the operator $\mathcal L_{{\A},0}$.
Even the strategy is in spirt of \cite{Z} and \cite{CT1,CT2},  we have to construct the argument adapted to $\mathcal{L}_{{\A},0}$, since there are many differences between two kinds of settings ( for example, $\mathcal L_{{\A},0}$ is a complex coefficient operator).

\subsection{Functional calculus} In this subsection, inspired by Cheeger-Taylor \cite{CT1,CT2}, we recall the functional calculus associated with the operator $\mathcal{L}_{{\A},0}$, see also \cite{FZZ}.

From \eqref{LAa} with $a=0$, we write
\begin{equation}\label{LAa-r}
\begin{split}
\mathcal{L}_{{\A},0}=-\partial_r^2-\frac{1}r\partial_r+\frac{L_{{\A},0}}{r^2},
\end{split}
\end{equation}
where the operator
\begin{equation}\label{L-angle}
\begin{split}
L_{{\A},0}&=(i\nabla_{\mathbb{S}^{1}}+{\A}(\hat{x}))^2,\qquad \hat{x}\in \mathbb{S}^1
\\&=-\Delta_{\mathbb{S}^{1}}+\big(|{\A}(\hat{x})|^2+i\,\mathrm{div}_{\mathbb{S}^{1}}{\A}(\hat{x})\big)+2i {\A}(\hat{x})\cdot\nabla_{\mathbb{S}^{1}}.
\end{split}
\end{equation}
Let $\hat{x}=(\cos\theta,\sin\theta)$, then
\begin{equation*}
\partial_\theta=-\hat{x}_2\partial_{\hat{x}_1}+\hat{x}_1\partial_{\hat{x}_2},\quad \partial_\theta^2=\Delta_{\mathbb{S}^{1}}.
\end{equation*}
Define $\alpha(\theta):[0,2\pi]\to \R$ such that
\begin{equation}\label{equ:alpha}
\alpha(\theta)={\bf A}(\cos\theta,\sin\theta)\cdot (-\sin\theta,\cos\theta),
\end{equation}
then by \eqref{eq:transversal}, we can write
\begin{equation*}
{\bf A}(\cos\theta,\sin\theta)=\alpha(\theta)(-\sin\theta,\cos\theta),\quad \theta\in[0,2\pi],
\end{equation*}
Thus, we obtain
\begin{equation}\label{LAa-s}
\begin{split}
L_{{\A},0}&=-\Delta_{\mathbb{S}^{1}}+\big(|{\A}(\hat{x})|^2+i\,\mathrm{div}_{\mathbb{S}^{1}}{\A}(\hat{x})\big)+2i {\A}(\hat{x})\cdot\nabla_{\mathbb{S}^{1}}\\
&=-\partial_\theta^2+\big(|{\alpha}(\theta)|^2+i\,{\alpha'}(\theta)\big)+2i {\alpha}(\theta)\partial_\theta\\
&=(i\partial_\theta+\alpha(\theta))^2.
\end{split}
\end{equation}
For simplicity, define the constant $\alpha$ to be
$$\alpha=\Phi_{\A}=\frac1{2\pi}\int_0^{2\pi} \alpha(\theta) d\theta.$$
Then the operator $i\partial_\theta+\alpha(\theta)$ with domain $H^1(\mathbb{S}^1)$ in $L^2(\mathbb{S}^1)$ has eigenvalue
$\nu(k)=k+\alpha, k\in\Z$ and the corresponding eigenfunction
\begin{equation}\label{eigf}
\varphi_k(\theta)=\frac1{\sqrt{2\pi}}e^{-i\big(\theta(k+\alpha)-\int_0^{\theta}\alpha(\theta') d\theta'\big)}.
\end{equation}
Therefore we obtain
$$L_{{\A},0}\varphi_k(\theta)=(k+\alpha)^2\varphi_k(\theta)$$
and we have the orthogonal decomposition $$L^2(\mathbb{S}^1)=\bigoplus_{k\in\Z}\mathcal{H}^k,$$
where $$\mathcal{H}^k=\text{span}\{\varphi_k(\theta)\}.$$
For $f\in L^2(\R^2)$, we can write $f$ into the form of separating variables
\begin{equation}\label{sep.v}
f(x)=\sum_{k\in\Z} c_{k}(r)\varphi_{k}(\theta)
\end{equation}
where
\begin{equation*}
 c_{k}(r)=\int_{0}^{2\pi}f(r,\theta)
\overline{\varphi_{k}(\theta)} d\theta.
\end{equation*}
Hence, on  each space $\mathcal{H}^{k}=\text{span}\{\varphi_k\}$, from \eqref{LAa-s}, we have
\begin{equation*}
\begin{split}
\LL_{{\A},0}=-\partial_r^2-\frac{1}r\partial_r+\frac{(k+\alpha)^2}{r^2}.
\end{split}
\end{equation*}
Let $\nu=\nu_k=|k+\alpha|$,  for $f\in L^2(\R^2)$, we define the Hankel transform of order $\nu$
\begin{equation}\label{hankel}
(\mathcal{H}_{\nu}f)(\rho,\theta)=\int_0^\infty J_{\nu}(r\rho)f(r,\theta) \,rdr,
\end{equation}
where the Bessel function of order $\nu$ is given by
\begin{equation}\label{Bessel}
J_{\nu}(r)=\frac{(r/2)^{\nu}}{\Gamma\left(\nu+\frac12\right)\Gamma(1/2)}\int_{-1}^{1}e^{isr}(1-s^2)^{(2\nu-1)/2} ds, \quad \nu>-1/2, r>0.
\end{equation}

%
%
%
%

Briefly recalling the functional calculus for well-behaved functions $F$ (see \cite{Taylor}),
\begin{equation}\label{funct}
F(\mathcal{L}_{{\A},0}) f(r_1,\theta_1)=\sum_{k\in\Z} \int_0^\infty \int_0^{2\pi} K(r_1,\theta_1,r_2,\theta_2) f(r_2,\theta_2)\; r_2\;dr_2\;d\theta_2
\end{equation}
where the kernel
$$K(r_1,\theta_1,r_2,\theta_2)=\sum_{k\in\Z}\varphi_{k}(\theta_1)\overline{\varphi_{k}(\theta_2)}K_{\nu_k}(r_1,r_2),$$
and
\begin{equation}\label{equ:knukdef}
  K_{\nu_k}(r_1,r_2)=\int_0^\infty F(\rho^2) J_{\nu_k}(r_1\rho)J_{\nu_k}(r_2\rho) \,\rho d\rho.
\end{equation}

\subsection{Schr\"odinger propagator}
In this subsection, we construct the propagator of Schr\"odinger equation.
The main result is the following.

\begin{theorem}[Schr\"odinger kernel]\label{thm:Sch-pro} Let $\LL_{{\A},0}$ be the operator in Theorem \ref{thm:dispersive} and suppose $x=r_1(\cos\theta_1,\sin\theta_1)$
and $y=r_2(\cos\theta_2,\sin\theta_2)$. Then, the kernel of Schr\"odinger propagator is given by
\begin{equation}\label{equ:heatkernel}
\begin{split}
  &e^{-it\LL_{{\A},0}}(x,y)\\
  =&\frac{e^{-\frac{|x-y|^2}{4it}} }{it}\frac{e^{i\int_{\theta_1}^{\theta_2}\alpha (\theta')d\theta'}}{4\pi^2}\big(\mathbbm{1}_{[0,\pi]}(|\theta_1-\theta_2|)
  +e^{-i2\pi\alpha}\mathbbm{1}_{[\pi,2\pi]}(|\theta_1-\theta_2|)\big)\\
  &-\frac{1}{4\pi^2}\frac{e^{-\frac{r_1^2+r_2^2}{4it}} }{it} e^{-i\alpha(\theta_1-\theta_2)+i\int_{\theta_2}^{\theta_{1}} \alpha(\theta') d\theta'} \int_0^\infty e^{-\frac{r_1r_2}{2it}\cosh s} \Big(\sin(|\alpha|\pi)e^{-|\alpha|s}\\
  &+\sin(\alpha\pi)\frac{(e^{-s}-\cos(\theta_1-\theta_2+\pi))\sinh(\alpha s)-i\sin(\theta_1-\theta_2+\pi)\cosh(\alpha s)}{\cosh(s)-\cos(\theta_1-\theta_2+\pi)}\Big) ds\\
  \triangleq&G(t;r_1,\theta_1,r_2,\theta_2)+D(t;r_1,\theta_1,r_2,\theta_2).
\end{split}
\end{equation}

\end{theorem}

\begin{remark} If $\alpha\in\Z$, then $D(t; r_1,\theta_1,r_2,\theta_2)$ vanishes. The first term becomes
$$(4\pi^2 it)^{-1}e^{-\frac{|x-y|^2}{4it}}, $$
which consists with the kernel of free Schr\"odinger propagator without potential.
\end{remark}

\begin{proof}
 From \eqref{eigf} and  \eqref{funct}, the kernel $e^{-it\LL_{{\A},0}}(x,y)$ is given by
  \begin{align}\label{equ:ktxyschr}
    K(t; r_1,\theta_1,r_2,\theta_2)=&\sum_{k\in\Z}\varphi_{k}(\theta_1)\overline{\varphi_{k}(\theta_2)}K_{\nu_k}(t; r_1,r_2)\\\nonumber
    =&\frac1{2\pi}e^{-i\alpha(\theta_1-\theta_2)+i\int_{\theta_2}^{\theta_{1}}\alpha(\theta') d\theta'} \sum_{k\in\Z} e^{-ik(\theta_1-\theta_2)}  K_{\nu_k}(t; r_1,r_2),
  \end{align}
  where $K_{\nu}$ is given by
\begin{align}\label{equ:knukdef12sch}
  K_{\nu}(t,r_1,r_2)=&\int_0^\infty e^{-it\rho^2}J_{\nu}(r_1\rho)J_{\nu}(r_2\rho) \,\rho d\rho\\\nonumber
  =&\lim_{\epsilon\searrow0}\int_0^\infty e^{-(\epsilon+it)\rho^2}J_{\nu}(r_1\rho)J_{\nu}(r_2\rho) \,\rho d\rho\\\nonumber
  =&\lim_{\epsilon\searrow0}\frac{e^{-\frac{r_1^2+r_2^2}{4(\epsilon+it)}}}{2(\epsilon+it)} I_\nu\Big(\frac{r_1r_2}{2(\epsilon+it)}\Big),
\end{align}
where we use the Weber identity, e.g. \cite[Proposition 8.7]{Taylor} and $I_\nu$ is the modified Bessel function.
For $z=\frac{r_1r_2}{2(\epsilon+it)}$ with $\epsilon>0$, we use the integral representation in \cite{Watson}  to write the modified Bessel function $I_\nu$:
\begin{equation}\label{m-bessel}
I_\nu(z)=\frac1{\pi}\int_0^\pi e^{z\cos s} \cos(\nu s) ds-\frac{\sin(\nu\pi)}{\pi}\int_0^\infty e^{-z\cosh s} e^{-s\nu} ds.
\end{equation}
Hence,  we need to consider
  \begin{equation}\label{equ:term1sch}
 \frac1{2\pi} \frac1\pi \sum_{k\in\Z} e^{-ik(\theta_1-\theta_2)} \int_0^\pi e^{z\cos s} \cos(\nu_k s) ds,
  \end{equation}
  and
    \begin{equation}\label{equ:term2sch}
 \frac1{2\pi} \frac1\pi \sum_{k\in\Z} e^{-ik(\theta_1-\theta_2)} \sin(\nu_k\pi)\int_0^\infty e^{-z\cosh s} e^{-s\nu_k} ds.
  \end{equation}
By the same argument as in \cite[Proposition 3.1]{FZZ}, we get
\begin{align*}
  \eqref{equ:term1sch}= & \frac1{2\pi^2}\times \begin{cases}
e^{z\cos(\theta_1-\theta_2)} e^{i(\theta_1-\theta_2)\alpha}\quad&\text{if}\quad |\theta_1-\theta_2|<\pi\\
e^{z\cos(\theta_1-\theta_2)}e^{i(\theta_1-\theta_2-2\pi)\alpha}\quad&\text{if}\quad \pi<|\theta_1-\theta_2|<2\pi,
 \end{cases}
\end{align*}
  and
\begin{align*}
 & \eqref{equ:term2sch}\\
 = &\frac1{2\pi^2}\int_0^\infty e^{-z\cosh s} \Big[\sin(|\alpha|\pi)e^{-|\alpha| s}\\
 &+\sin(\alpha\pi)\frac{(e^{- s}-\cos(\theta_1-\theta_2+\pi))\sinh(\alpha s)
-i\sin(\theta_1-\theta_2+\pi)\cosh(\alpha s)}{\cosh( s)-\cos(\theta_1-\theta_2+\pi)}\Big]\,ds.
\end{align*}
  Plugging these into \eqref{equ:ktxyschr} and let $\epsilon\to 0^+$, we obtain Theorem \ref{thm:Sch-pro}.

  \end{proof}

By using Theorem \ref{thm:Sch-pro} and the same proof as in \cite[Proposition 3.2]{FZZ}, we can show
\begin{corollary}[Dispersive estimate]\label{cor:shcrpoint} There holds
\begin{equation}\label{equ:dispes}
  \big\|e^{-it\LL_{{\A},0}}f\big\|_{L^\infty(\R^2)}\lesssim |t|^{-1}\|f\|_{L^1(\R^2)}.
\end{equation}
\end{corollary}

\subsection{Resolvent kernel} In this subsection, we use the Schr\"odinger operator to construct the resolvent kernel
which will be used to prove spectral measure theorem according to Stone's formula.
We first note that for $z\in\{z\in\C:\;{\rm Im}(z)>0\}$
\begin{equation*}
  (s-z)^{-1}=-\frac1i\int_0^\infty e^{-ist}e^{izt}\;dt,\quad \forall\;s\in\R,
\end{equation*}
then we obtain for $z=\lambda^2+i\epsilon$ with $\epsilon>0$
\begin{equation}\label{equ:resfor}
  \big(\LL_{{\A},0}-(\lambda^2+i0)\big)^{-1}=-\frac1i\lim_{\epsilon\searrow0}\int_0^\infty e^{-it\LL_{{\A},0}}e^{it(\lambda^2+i\epsilon)}\;dt.
\end{equation}


From Theorem \ref{thm:Sch-pro}, \eqref{A-al} and \eqref{B-al},
we need to consider
\begin{equation}\label{equ:resG}
\begin{split}
&\frac1i\lim_{\epsilon\searrow0}\int_0^\infty G(t;r_1,\theta_1,r_2,\theta_2) e^{it(\lambda^2+i\epsilon)}\;dt\\
=& \frac1i\lim_{\epsilon\searrow0}\int_0^\infty \frac{e^{-\frac{|x-y|^2}{4it}} }{it} e^{it(\lambda^2+i\epsilon)}\;dt \times\,A_\alpha(\theta_1,\theta_2),
\end{split}
\end{equation}
and
\begin{equation}\label{equ:resD}
\begin{split}
&\frac1i\lim_{\epsilon\searrow0}\int_0^\infty D(t;r_1,\theta_1,r_2,\theta_2) e^{it(\lambda^2+i\epsilon)}\;dt\\
=&\frac1i\lim_{\epsilon\searrow0} \int_0^\infty \int_0^\infty\frac{e^{-\frac{r_1^2+r_2^2+2r_1r_2\cosh s}{4it}} }{it} e^{it(\lambda^2+i\epsilon)}\;dt B_\alpha(s, \theta_1,\theta_2)\, ds.
\end{split}
\end{equation}
On the one hand, note that
$$\int_{\R^2}e^{-ix\cdot\xi} e^{-it|\xi|^2}\;d\xi=\frac{\pi}{it}e^{-\frac{|x|^2}{4it}},$$
we get for $z=\lambda^2+i\epsilon$ with $\epsilon>0$
\begin{align*}
 & \int_0^\infty \frac{e^{-\frac{|x-y|^2}{4it}} }{it}e^{itz}\;dt
  =\frac{1}{\pi}\int_0^\infty\int_{\R^2}e^{-i(x-y)\cdot\xi} e^{-it|\xi|^2}\;d\xi \,e^{itz}\;dt\\
  =&\frac{1}{\pi}\int_{\R^2}e^{-i(x-y)\cdot\xi} \int_0^\infty e^{-it(|\xi|^2-z)}\;dt\;d\xi
  =\frac{1}{i\pi}\int_{\R^2}\frac{e^{-i(x-y)\cdot\xi}}{|\xi|^2-z}\;d\xi.
\end{align*}
On the other hand, we similarly obtain
\begin{equation}\label{equ:intr1r223}
  \int_0^\infty  \frac{e^{-\frac{r_1^2+r_2^2}{4it}} }{it} e^{-\frac{r_1r_2}{2it}\cosh s} e^{itz}\;dt=\frac{1}{i\pi}\int_{\R^2}\frac{e^{-i{\bf n}\cdot\xi}}{|\xi|^2-z}\;d\xi,
\end{equation}
where ${\bf n}=(r_1+r_2, \sqrt{2r_1r_2(\cosh s-1)})$.\vspace{0.2cm}

Therefore we finally obtain
\begin{proposition}[Outgoing resolvent kernel]\label{prop:res-ker-out}
Let $x=r_1(\cos\theta_1,\sin\theta_1)$ and $y=r_2(\cos\theta_2,\sin\theta_2)$,
then we have the expression of resolvent kernel
\begin{align}\label{equ:res-ker-out}
     \big(\LL_{{\A},0}-(\lambda^2+i0)\big)^{-1}=&\frac1{\pi}\int_{\R^2}\frac{e^{-i(x-y)\cdot\xi}}{|\xi|^2-(\lambda^2+i0)}\;d\xi\, A_\alpha(\theta_1,\theta_2)\\\nonumber
  &+\frac1{\pi}\int_0^\infty \int_{\R^2}\frac{e^{-i{\bf n}\cdot\xi}}{|\xi|^2-(\lambda^2+i0)}\;d\xi \, B_\alpha(s,\theta_1,\theta_2)\;ds,
\end{align}
where $A_\alpha(\theta_1,\theta_2)$, $B_\alpha(s, \theta_1,\theta_2)$ are respectively given in  \eqref{A-al} and \eqref{B-al}, and  ${\bf n}=(r_1+r_2, \sqrt{2r_1r_2(\cosh s-1)})$.
\end{proposition}

 On the one hand, from \eqref{LAa-r} and \eqref{LAa-s}, we see
 $\overline{\mathcal L_{{\A},0}}$ is the same as $\mathcal L_{{\A},0}$ with replacing $\alpha(\theta)$ by $-\alpha(\theta)$.
 On the other hand, we note that
 \begin{equation}\label{out-inc}
\begin{split}
(\mathcal L_{{\A},0}-(\lambda^2-i0))^{-1}=\overline{(\overline{\mathcal L_{{\A},0}}-(\lambda^2+i0))^{-1}},
\end{split}
\end{equation}
and the fact $\overline{A_{-\alpha}}=A_\alpha$ and $\overline{B_{-\alpha}}=B_\alpha$, hence we similarly obtain
\begin{proposition}[Incoming resolvent kernel]\label{prop:res-ker-inc}
Let $x=r_1(\cos\theta_1,\sin\theta_1)$ and $y=r_2(\cos\theta_2,\sin\theta_2)$,
then we have the expression of resolvent kernel
\begin{align}\label{equ:res-ker-inc}
     \big(\LL_{{\A},0}-(\lambda^2-i0)\big)^{-1}=&\frac1{\pi}\int_{\R^2}\frac{e^{-i(x-y)\cdot\xi}}{|\xi|^2-(\lambda^2-i0)}\;d\xi\, A_\alpha(\theta_1,\theta_2)\\\nonumber
  &+\frac1{\pi}\int_0^\infty \int_{\R^2}\frac{e^{-i{\bf n}\cdot\xi}}{|\xi|^2-(\lambda^2-i0)}\;d\xi \, B_\alpha(s,\theta_1,\theta_2)\;ds,
\end{align}
where $A_\alpha(\theta_1,\theta_2)$, $B_\alpha(s, \theta_1,\theta_2)$ are respectively given in  \eqref{A-al} and \eqref{B-al}, and  ${\bf n}=(r_1+r_2, \sqrt{2r_1r_2(\cosh s-1)})$.
\end{proposition}

\subsection{Spectral measure kernel} We prove the main result Proposition \ref{prop:spect} in this subsection.
According to Stone's formula, the spectral measure is related to the resolvent
\begin{equation}\label{equ:spemes}
   dE_{\sqrt{\LL_{\A},0}}(\lambda)=\frac{d}{d\lambda}dE_{\sqrt{\LL_{\A},0}}(\lambda)\;d\lambda
   =\frac{\lambda}{i\pi}\big(R(\lambda+i0)-R(\lambda-i0)\big)\;d\lambda
\end{equation}
where the resolvent
$$R(\lambda\pm i0)=\lim_{\epsilon\searrow0}\big(\LL_{{\A},0}-(\lambda^2\pm i\epsilon)\big)^{-1}.$$
From Proposition \ref{prop:res-ker-out} and Proposition \ref{prop:res-ker-inc}, we obtain
\begin{align*}
  &dE_{\sqrt{\LL_{\A},0}}(\lambda; x,y)\\
  = &\frac1{\pi} \frac{\lambda}{i\pi} A_\alpha(\theta_1,\theta_2)\int_{\R^2} e^{-i(x-y)\cdot\xi}\Big(\frac{1}{|\xi|^2-(\lambda^2+i0)}-\frac{1}{|\xi|^2-(\lambda^2-i0)}\Big)\;d\xi\\
  &+\frac1{\pi} \frac{\lambda}{i\pi} \int_0^\infty \Big[ \int_{\R^2} e^{-i\bf{n}\cdot\xi}\Big(\frac{1}{|\xi|^2-(\lambda^2+i0)}-\frac{1}{|\xi|^2-(\lambda^2-i0)}\Big)\;d\xi \Big]B_\alpha(s,\theta_1,\theta_2)\;ds.
\end{align*}

On the one hand, we note that
\begin{equation}\label{id-spect}
\begin{split}
&\lim_{\epsilon\to 0^+}\frac{\lambda}{i\pi }\int_{\R^2} e^{-ix\cdot\xi}\Big(\frac{1}{|\xi|^2-(\lambda^2+i\epsilon)}-\frac{1}{|\xi|^2-(\lambda^2-i\epsilon)}\Big) d\xi\\
=&\lim_{\epsilon\to 0^+} \frac{\lambda}{\pi }\int_{\R^2} e^{-ix\cdot\xi}\Im\Big(\frac{1}{|\xi|^2-(\lambda^2+i\epsilon)}\Big)d\xi\\
=&\lim_{\epsilon\to 0^+} \frac{\lambda}{\pi }\int_{0}^\infty \frac{\epsilon}{(\rho^2-\lambda^2)^2+\epsilon^2} \int_{|\omega|=1} e^{-i\rho x\cdot\omega} d\sigma_\omega  \, \rho d\rho\\
=& \lambda \int_{|\omega|=1} e^{-i\lambda x\cdot\omega} d\sigma_\omega  \\
\end{split}
\end{equation}
where we use the fact that
 the  Poisson kernel is an approximation to the identity which implies that, for any reasonable function $m(x)$
\begin{equation}
\begin{split}
m(x)&=\lim_{\epsilon\to 0^+}\frac1\pi \int_{\R} {\rm Im}\Big(\frac{1}{x-(y+i\epsilon)}\Big) m(y)dy
\\&=\lim_{\epsilon\to 0^+}\frac1\pi \int_{\R} \frac{\epsilon}{(x-y)^2+\epsilon^2} m(y)dy.
\end{split}
\end{equation}
On the other hand, from \cite[Theorem 1.2.1]{sogge}, we also note that
\begin{equation}
\begin{split}
\int_{\mathbb{S}^{1}} e^{-i x\cdot\omega} d\sigma(\omega)=\sum_{\pm}  a_\pm(|x|) e^{\pm i|x|}
\end{split}
\end{equation}
where
\begin{equation}
\begin{split}
| \partial_r^k a_\pm(r)|\leq C_k(1+r)^{-\frac{1}2-k},\quad k\geq 0.
\end{split}
\end{equation}
Therefore we obtain that
 \begin{equation*}
 \begin{split}
 dE_{\sqrt{\LL_{{\A},0}}}(\lambda;x,y) =&
\frac{\lambda}{\pi} \sum_{\pm}\Big(
 a_\pm(\lambda |x-y|)e^{\pm i\lambda |x-y|} A_{\alpha}(\theta_1,\theta_2)
 \\&\quad
+\int_0^\infty a_\pm(\lambda |{\bf{n}}|)e^{\pm i\lambda |{\bf n}|}
B_{\alpha}(s,\theta_1,\theta_2) ds\Big).
\end{split}
\end{equation*}
Notice $d=|x-y|$ and $|{\bf n}|=d_s$ in \eqref{d-j} and \eqref{d-s} again, we prove Proposition \ref{prop:spect}.

\section{The proof of dispersive estimates}

In this section, we use stationary phase argument and Proposition
\ref{prop:spect} to prove Theorem \ref{thm:dispersive}.  We only prove \eqref{dis-kg} since
\eqref{dis-S} have been proved in Corollary \ref{cor:shcrpoint} and \eqref{dis-w} follows from
Theorem \ref{thm:dispersive} and the stationary phase argument of  \cite[Theorem 3.3]{Z} or \cite[Proposition 3.3]{Zhang}.
Indeed, it is more complicate for the Klein-Gordon than the wave and Schr\"odinger case,
the reason is that the Klein-Gordon equation  likes wave at high frequency but Schr\"odinger at low frequency.

In this section, recalling $\phi$  as in \eqref{dp}, we know that $\phi\in C_0^\infty(\mathbb{R}\setminus\{0\})$ take values in
$[0,1]$ and be supported in $[1/2,2]$ such that
\begin{equation}\label{dp'}
1=\sum_{j\in\Z}\phi(2^{-j}\lambda),\quad\lambda>0\quad \text{and}\quad \varphi_0(\lambda)=\sum_{j\leq 0}\phi(2^{-j}\lambda).
\end{equation}

\begin{proposition}[Dispersive estimates for low frequency]\label{dispersive-l}
Let $dE_{\sqrt{\LL_{{\A},0}}}(\lambda;x,y)$ be the spectral measure in Proposition \ref{prop:spect}. There exists  a constant $C$ independent of points $x,y\in\R^2$ such that the kernel
of Klein-Gordon propagator satisfies
\begin{equation}\label{disper-l}
\Big|\int_0^\infty e^{it\sqrt{1+\lambda^2}} \varphi_0(\lambda) dE_{\sqrt{\LL_{{\A},0}}}(\lambda;x,y) \Big|\leq C (1+|t|)^{-1}.
\end{equation}

\end{proposition}

\begin{remark}
The decay rate $O(1+|t|)^{-1}$ for Klein-Gordon at low frequency is same as Schr\"odinger's decay.
\end{remark}

To prove Proposition \ref{dispersive-l}, we need the following lemma:

\begin{lemma} Let $A_{\alpha}(\theta_1,\theta_2)$ and $B_{\alpha}(s,\theta_1,\theta_2)$ in  \eqref{A-al} and \eqref{B-al},
then there exists a constant $C$ such that
\begin{equation}\label{est:AB}
|A_{\alpha}(\theta_1,\theta_2)|+\int_0^\infty \big|B_{\alpha}(s,\theta_1,\theta_2)\big| ds\leq C.
\end{equation}
\end{lemma}

\begin{proof} This is a consequence of direct computation.
\end{proof}

\begin{proof} [{\bf The proof of Proposition \ref{dispersive-l}}:] By using Proposition \ref{prop:spect},
we write the kernel
\begin{align*}
   & \int_0^\infty e^{it\sqrt{1+\lambda^2}} \varphi_0(\lambda) dE_{\sqrt{\LL_{{\A},0}}}(\lambda;x,y)\\
=&\frac{A_{\alpha}(\theta_1,\theta_2)}{\pi}\sum_{\pm}\int_0^\infty e^{it\sqrt{1+\lambda^2}} \varphi_0(\lambda)\lambda \,a_\pm(\lambda |x-y|)e^{\pm i\lambda |x-y|} \;d\lambda\\
&+\frac{1}{\pi}\sum_{\pm}\int_0^\infty e^{it\sqrt{1+\lambda^2}} \varphi_0(\lambda)\lambda \int_0^\infty a_\pm(\lambda |{\bf{n}}|)e^{\pm i\lambda |{\bf n}|}
B_{\alpha}(s,\theta_1,\theta_2) ds\;d\lambda\\
=:& I+II,
\end{align*}
where ${\bf n}=(r_1+r_2, \sqrt{2r_1r_2(\cosh s-1)})$
and
$$| \partial_r^m a_\pm(r)|\leq C_m(1+r)^{-\frac{1}2-m},\quad m\geq 0,$$
which implies
\begin{equation}\label{equ:akrder}
\big|\lambda^m \partial_\lambda^m \big(a_\pm(\lambda r)\big)\big|\leq C_m(1+\lambda r)^{-\frac{1}2},\quad m\geq 0.
\end{equation}
When $|t|\lesssim 1$, from \eqref{est:AB} and \eqref{equ:akrder}, it is easy to show \eqref{disper-l} due to the compact support of $\varphi_0$. From now on, we only need to consider
the case $t\gg1$ by symmetry.
Therefore, it suffices to show, for all $0<r$ and $t\gg1$
\begin{equation}\label{equ:mdisred}
  \Big|\int_0^\infty e^{i(t\sqrt{1+\lambda^2}\pm \lambda r)} \varphi_0(\lambda)\lambda a_\pm(\lambda r) \;d\lambda\Big|\lesssim t^{-1}.
\end{equation}
Indeed, we apply  \eqref{equ:mdisred} with $r=|x-y|$ to show
$|I|\lesssim t^{-1}$,
and use  \eqref{equ:mdisred} with $r=r_s=|{\bf n}|$ and \eqref{est:AB}  to obtain
\begin{equation*}
\begin{split}
| II|\lesssim & \Big|\int_0^\infty e^{it\sqrt{1+\lambda^2}} \varphi_0(\lambda)\lambda \int_0^\infty a_\pm(\lambda r_s)e^{\pm i\lambda r_s}
B_{\alpha}(s,\theta_1,\theta_2) ds\;d\lambda\Big|\\
\lesssim& t^{-1} \int_0^\infty
|B_{\alpha}(s,\theta_1,\theta_2)| ds \lesssim t^{-1}.
\end{split}
\end{equation*}
To prove \eqref{equ:mdisred} when $r\geq 4t$, we are further reduced to show
\begin{equation}\label{equ:mdisred1}
  \sum_{j\leq 0}\Big|\int_0^\infty e^{i(t\sqrt{1+\lambda^2}\pm \lambda r)} \phi(2^{-j}\lambda)\lambda a_\pm(\lambda r) \;d\lambda\Big|\lesssim t^{-1},\quad\forall\,r\geq 4t,\;t\gg1.
\end{equation}
By changing variable, we obtain
\begin{align*}
I_j(t,r):=&\int_0^\infty e^{i(t\sqrt{1+\lambda^2}\pm \lambda r)} \phi(2^{-j}\lambda)\lambda a_\pm(\lambda r) \;d\lambda\\
=&2^{2j}\int_0^\infty e^{i2^jt(\sqrt{2^{-2j}+\lambda^2}\pm\frac{\lambda r}{t})}\phi(\lambda)\lambda a_\pm(2^j\lambda r) \;d\lambda\\
=&2^{2j}\int_0^\infty e^{i2^jt\Phi_{j,\pm}(\lambda,t,r)}\phi(\lambda)\lambda a_\pm(2^j\lambda r) \;d\lambda,
\end{align*}
where \begin{equation}
\Phi_{j,\pm}(\lambda,t,r)=\sqrt{2^{-2j}+\lambda^2}\pm\frac{\lambda r}{t}.
\end{equation}
By direct computation, we have
$$\pa_\lambda \Phi_{j,\pm}(\lambda,t,r)=\frac{\lambda}{\sqrt{2^{-2j}+\lambda^2}}\pm\frac{r}{t},$$
then we further estimate, for $j\leq0$ and $r\geq 4t$
$$\big|\pa_\lambda \Phi_{j,\pm}\big|\geq\frac12\quad\text{and}\quad \pa_\lambda^2 \Phi_{j,\pm}=\frac{2^{-2j}}{(2^{-2j}+\lambda^2)^{\frac32}}>0.$$

We recall the following Van der Corput lemma, see \cite{Stein}.
\begin{lemma}[Van der Corput] Let $\phi$ be real-valued and smooth in $(a,b)$, and that $|\phi^{(k)}(x)|\geq1$ for all $x\in (a,b)$. Then
\begin{equation}
\left|\int_a^b e^{i\lambda\phi(x)}\psi(x)dx\right|\leq c_k\lambda^{-1/k}\left(|\psi(b)|+\int_a^b|\psi'(x)|dx\right)
\end{equation}
holds when (i) $k\geq2$ or (ii) $k=1$ and $\phi'(x)$ is monotonic. Here $c_k$ is a constant depending only on $k$.
\end{lemma}
And so, by Van-der Corput Lemma   and \eqref{equ:akrder}, one has
$$|I_j(t,r)|\lesssim 2^{2j}(2^jt)^{-1}\int_{\frac12}^2\big|\pa_\lambda\big[\phi(\lambda)\lambda a_\pm(2^j\lambda r)  \big]\big|\,d\lambda\lesssim 2^jt^{-1}.$$
Next, we turn to consider the case that  $\frac{t}4\leq r\leq4t.$ It is easy to see that
$$|\partial_\lambda^2\Phi_{j,\pm}|=\frac{2^{-2j}}{(2^{-2j}+\lambda^2)^\frac32}\gtrsim 2^j.$$
Hence,  by Van-der Corput Lemma and \eqref{equ:akrder} again and the fact $t\sim r$, we obtain
\begin{align*}
  |I_j(t,r)|\lesssim  & 2^{2j}(2^{2j}t)^{-\frac12}\int_{\frac12}^2\big|\pa_\lambda\big[\phi(\lambda)\lambda a_\pm(2^j\lambda r)  \big]\big|\,d\lambda\\
  \lesssim&2^jt^{-\frac12}(2^jr)^{-\frac12}\lesssim 2^{j/2}t^{-1}.
  \end{align*}
Therefore,
$$\sum_{j\leq0}|I_j(t,r)|\lesssim t^{-1}\sum_{j\leq0} 2^{j/2}\lesssim t^{-1}.$$
And so \eqref{equ:mdisred1} follows. \\
Finally, we prove \eqref{equ:mdisred} when $r\leq\frac{t}4$. By variable changes, it suffices to prove
\begin{equation}\label{equ:mdisred'}
  \Big|\int_0^\infty e^{i(\sqrt{t^2+t\lambda^2}\pm \lambda \tilde{p})} \varphi_0(\lambda/\sqrt{t})\lambda a_\pm(\lambda \tilde{p}) \;d\lambda\Big|\lesssim 1
\end{equation}
where $\tilde{r}=r/\sqrt{t}$. We divide it into two pieces and we are reduced to prove
\begin{equation}\label{equ:mdisred'-1}
  \Big|\int_0^\infty e^{i(\sqrt{t^2+t\lambda^2}\pm \lambda \tilde{r})} \varphi_0(\lambda/\sqrt{t})\lambda a_\pm(\lambda \tilde{r}) \varphi_0(\lambda) \;d\lambda\Big|\lesssim 1
\end{equation}
and
\begin{equation}\label{equ:mdisred'-2}
  \Big|\sum_{m\geq 1}\int_0^\infty e^{i(\sqrt{t^2+t\lambda^2}\pm \lambda \tilde{r})} \varphi_0(\lambda/\sqrt{t})\lambda a_\pm(\lambda \tilde{r})\phi(\frac{\lambda}{2^m}) \;d\lambda\Big|\lesssim 1.
\end{equation}
It is easy to prove the \eqref{equ:mdisred'-1} due to $\lambda\leq 2$ and \eqref{equ:akrder}. In the rest of the proof,  we prove \eqref{equ:mdisred'-2}.
We first consider \eqref{equ:mdisred'-2} with $+$ sign. In this case,
$$
e^{i\sqrt{t^2+t\lambda^2} + i \tilde{r} \lambda}= (L^+)^N (e^{i\sqrt{t^2+t\lambda^2} + i \tilde{r} \lambda}), \quad L^+ = \frac1i\big(\frac{t\lambda}{\sqrt{t^2+t\lambda^2}} + \tilde{r}\big)^{-1} \frac{\partial}{\partial \lambda}.
$$
Due to the support of $\varphi_0$, it gives $0<\lambda<2\sqrt{t}$. Then  for $\ell\geq0$ and $t\gg1$, it follows
\begin{equation}\label{I-bp4}
\begin{split}
\Big|\partial_\lambda^{\ell}\Big[\big(\frac{t\lambda}{\sqrt{t^2+t\lambda^2}}+\tilde{r}\big)^{-1}\Big]\Big|\leq C_\ell \lambda^{-1-\ell}.
\end{split}
\end{equation}
By using \eqref{equ:akrder} and \eqref{I-bp4} with $\ell=2$,
we obtain
$$
\text{LHS of} \,\eqref{equ:mdisred'-2} \lesssim \sum_{m \geq 1} \int_{\lambda \sim 2^m} \lambda^{-2} \, d\lambda \leq C.
$$
Now we  treat  \eqref{equ:mdisred'-2} with $-$ sign. Let $\Phi(\lambda,\tilde{r})=\sqrt{t^2+t\lambda^2}-\tilde{r}\lambda$.
We divide \eqref{equ:mdisred'-2} into two pieces:
\begin{equation*}
\begin{split}
II_1=&\Big|\sum_{m\geq 1}\int_0^\infty e^{i\Phi(\lambda,\tilde{r})} \varphi_0(\lambda/\sqrt{t})\lambda a_\pm(\lambda \tilde{r})\phi\big(\frac{\lambda}{2^m}\big)
\varphi_0(8\tilde{r} \lambda) d\lambda\Big|, ~\\II_2=&\Big|\int_0^\infty e^{i\Phi(\lambda,\tilde{r})} \varphi_0(\lambda/\sqrt{t})\lambda a_\pm(\lambda \tilde{r})
\left(1-\varphi_0(\lambda)\right) \big( 1 - \varphi_0(8\tilde{r} \lambda)
\big) d\lambda\Big|.
\end{split}
\end{equation*}
 We first
consider $II_1$. The integrand in $II_1$ vanishes
when $\tilde{r}>1/8$ due to the supports of $\phi$ and $\varphi_0$ (which implies $\lambda \leq (8 \tilde{r})^{-1}$ and $\lambda \geq
1$). Thus we only consider $1\leq \lambda<2\sqrt{t}$ and $\tilde{r}\leq1/8$, therefore $|\partial_\lambda\Phi| =
\frac{t\lambda}{\sqrt{t^2+t\lambda^2}} - \tilde{r} \geq \frac1{\sqrt{5}}\lambda- \tilde{r} \geq\frac1{100}\lambda$. As in \eqref{I-bp4},
on the support of $\varphi_0(\lambda/\sqrt{t})$, for $\ell\geq0$ and $t\gg1$, we also use the induction argument to obtain
\begin{equation}\label{I-bp4'}
\begin{split}
\Big|\partial^\ell_\lambda\Big[\big(\frac{t\lambda}{\sqrt{t^2+t\lambda^2}}-\tilde{r}\big)^{-1}\Big]\Big|\leq C_\ell \Big(\frac{t\lambda}{\sqrt{t^2+t\lambda^2}}-\tilde{r}\Big)^{-1}\lambda^{-\ell}\leq C_\ell\lambda^{-1-\ell}.
\end{split}
\end{equation}
 Define the operator
$L=L(\lambda,\tilde{r})=(\frac{t\lambda}{\sqrt{t^2+t\lambda^2}}-\tilde{r} )^{-1}\partial_\lambda$. By using
 \eqref{equ:akrder} and integration by parts again, we obtain
\begin{equation*}
\begin{split}
II_1\leq&\sum_{m\geq1}\Big|\int_0^{\infty} L^{N}
\big(e^{i\Phi(\lambda,\tilde{r})} \big)\Big[\varphi_0(\lambda/\sqrt{t})\lambda a_\pm(\lambda \tilde{r})\phi(\frac{\lambda}{2^m})
\varphi_0(8\tilde{r} \lambda)  \Big]d\lambda\Big|
\\\leq &C_N\sum_{m\geq1}\int_{\lambda\sim
2^{m}}\lambda^{1-2N}d\lambda\leq C_N.
\end{split}
\end{equation*}
Finally we estimate $II_2$. Based on the size of $\partial_\lambda \Phi$, we make a further decomposition of $II_2$
\begin{equation*}
\begin{split}
II_2\leq &\Big|\int_0^{\infty} e^{i\Phi(\lambda,\tilde{r})} \varphi_0(\lambda/\sqrt{t})\lambda a_\pm(\lambda \tilde{r})
\left(1-\varphi_0(\lambda)\right)\varphi_0(\frac{t\lambda}{\sqrt{t^2+t\lambda^2}}-\tilde{r}) \big( 1 - \varphi_0(8\tilde{r} \lambda) \big) \, d\lambda\Big|\\
&+\sum_{m\geq1}\Big|\int_0^{\infty}
e^{i\Phi(\lambda,\tilde{r})} \varphi_0(\lambda/\sqrt{t})\lambda a_\pm(\lambda \tilde{r})
\big(1-\varphi_0(\lambda)\big)\phi\big(\frac{\frac{t\lambda}{\sqrt{t^2+t\lambda^2}}-\tilde{r}}{2^m}\big)\big( 1 - \varphi_0(8\tilde{r} \lambda) \big) \, d\lambda\Big|\\:=&II_2^1+II_2^2.
\end{split}
\end{equation*}
Due to  the compact support of the second $\varphi_0$ factor in $II_2^1$, one has
\begin{equation}\label{II_2}
\Big|\frac{t\lambda}{\sqrt{t^2+t\lambda^2}}-\tilde{r}\Big|\leq 1.
\end{equation}
If $\tilde{r} \leq 10$, from $\lambda<2\sqrt{t}$ again, we must have $\lambda \leq 100$ otherwise the integrand of $II_2^1$ vanishes.
Then we see that $II_2^1$ is uniformly bounded. If $\tilde{r} \geq
10$, from \eqref{II_2} and $\lambda<2\sqrt{t}$, we have $\tilde{r}\sim\lambda$.
Hence,  by letting $\lambda'=\lambda/\sqrt{1+\lambda^2}$ and using \eqref{equ:akrder} with $\alpha = 0$, it follows that
\begin{equation*}
\begin{split}
II_2^1&\le\int_{\{\lambda<2\sqrt{t}:|\frac{t\lambda}{\sqrt{t^2+t\lambda^2}}-\tilde{r}|\leq 1\}}\lambda^{n-1}(1+\tilde{r}\lambda)^{-\frac{n-1}2}d\lambda
\\&\leq C \sqrt{t}\int_{\{\lambda<2:|\frac{\lambda}{\sqrt{1+\lambda^2}}-\frac{\tilde{r}}{\sqrt{t}}|\leq 1/\sqrt{t}\}}d\lambda\\&
\leq C \sqrt{t} \int_{\{\lambda<2:|\lambda'-\frac{\tilde{r}}{\sqrt{t}}|\leq 1/\sqrt{t}\}}(1+\lambda^2)^{3/2}d\lambda'\leq C.
\end{split}
\end{equation*}
Now we consider $II_2^2$. We estimate \begin{equation*}
\begin{split}
II_2^2\leq&\sum_{m\geq1}\Big|\int_0^\infty L^{N}
\big(e^{i\Phi(\lambda,\tilde{r})}\big) \varphi_0(\lambda/\sqrt{t})\lambda a_\pm(\lambda \tilde{r})\\&\qquad \qquad\big(1-\varphi_0(\lambda)\big)\phi\big(2^{-m}(\frac{t\lambda}{\sqrt{t^2+t\lambda^2}}-\tilde{r})\big)\big( 1 - \varphi_0(8\tilde{r} \lambda) \big) \, d\lambda\Big|.
\end{split}
\end{equation*}
Let $$b(\lambda)=\lambda a_\pm(\lambda \tilde{r})\big(1-\varphi_0(\lambda)\big)\phi\big(2^{-m}(\frac{t\lambda}{\sqrt{t^2+t\lambda^2}}-\tilde{r})\big)\big( 1 - \varphi_0(8\tilde{r} \lambda) \big),$$
then on the support of $b$ with $\lambda\geq 1/2$, we use \eqref{equ:akrder} to obtain
$$|\partial_\lambda ^\alpha b|\leq C_{\alpha}\lambda (1+\tilde{r}\lambda)^{-1/2}.$$
Hence from the first inequality of \eqref{I-bp4'}, we obtain
\begin{equation*}
\begin{split}
|(L^*)^N [b(\lambda)]|&\leq C_N 2^{-mN}\lambda (1+\tilde{r}\lambda)^{-1/2}.
\end{split}
\end{equation*}
Therefore we use integration by parts to obtain
\begin{equation*}
\begin{split}
II_2^2 \leq C_N\sum_{m\geq1}2^{-mN}\int_{\{\lambda<2\sqrt{t},|\frac{t\lambda}{\sqrt{t^2+t\lambda^2}}-\tilde{r}|\sim 2^m\}}\lambda(1+\tilde{r}\lambda)^{-\frac12}d\lambda.
\end{split}
\end{equation*}
If $\tilde{r}\leq 2^{m+1}$, since $|\frac{t\lambda}{\sqrt{t^2+t\lambda^2}}-\tilde{r}|\sim 2^m$, then $\lambda \leq 2^{m+2}$. One has
\begin{equation*}
\begin{split}
II_2^2\leq C_N\sum_{m\geq 1}2^{-mN}2^{(m+2)2}\leq C.
\end{split}
\end{equation*}
If $\tilde{r}\geq 2^{m+1}$, we
have $\lambda\sim \tilde{r}$, thus we choose $N$ large enough such that
\begin{equation*}
\begin{split}
II_2^2\leq C_N t^{1/2}  \sum_{m\geq1}2^{-mN} \int_{\{\lambda<2:|\frac{\lambda}{\sqrt{1+\lambda^2}}-\frac{\tilde{r}}{\sqrt{t}}|\sim \frac{2^m}{\sqrt{t}}\}}d\lambda
\leq C_N\sum_{m\geq1}2^{-mN}2^{m}\leq C
\end{split}
\end{equation*}
which concludes the proof of
\eqref{equ:mdisred'}.
Thus, we complete the proof of  Proposition \ref{dispersive-l}.

\end{proof}\vspace{0.2cm}

\begin{proposition}[Dispersive estimates for high frequency]\label{dispersive-h}
 For all integers $k\geq1$, the kernel estimate
\begin{equation}\label{disper}
\begin{split}
&\Big|\int_0^\infty e^{it\sqrt{1+\lambda^2}} \phi(2^{-k}\lambda)  dE_{\sqrt{\LL_{{\A},0}}}(\lambda;x,y)
\Big|\\ \leq& C 2^{k\frac{3+\eta}2}\left(2^{-k}+|t|\right)^{-\frac{1+\eta}2}.
\end{split}
\end{equation}
holds for $0\leq \eta\leq 1$ and a constant $C$ independent of $k$, $t$ and points $x,y\in \R^2$.
\end{proposition}

\begin{proof}
By Proposition \ref{prop:spect}, similarly as above,
we are reduced to show
\begin{equation}\label{equ:redshowhaig}
  \Big|\int_0^\infty e^{i(t\sqrt{1+\lambda^2}\pm \lambda r)}  \phi(2^{-k}\lambda)\lambda a_\pm(\lambda r) \;d\lambda\Big|\lesssim
  2^{k\frac{3+\eta}2}\left(2^{-k}+|t|\right)^{-\frac{1+\eta}2},\quad\forall\,r>0.
\end{equation}
When $|t|\leq 2^{-k},$ \eqref{equ:redshowhaig} follows by the compact support of $\phi$. Hence, we only need to consider the case that $t>2^{-k}$ by symmetry. By changing variable,
it gives
$$\int_0^\infty e^{i(t\sqrt{1+\lambda^2}\pm \lambda r)}  \phi(2^{-k}\lambda)\lambda a_\pm(\lambda r) \;d\lambda
=2^{2k}\int_0^\infty e^{i2^kt(\sqrt{2^{-2k}+\lambda^2}\pm \frac{\lambda r}{t})}  \phi(\lambda)\lambda a_\pm(2^k\lambda r) \;d\lambda.$$
  Hence, it suffices to show for any $0\leq\eta\leq1$
\begin{equation}\label{equ:highfreds}
  \Big|\int_0^\infty e^{i2^kt(\sqrt{2^{-2k}+\lambda^2}\pm \frac{\lambda r}{t})}  \phi(\lambda)\lambda a_\pm(2^k\lambda r) \;d\lambda\Big|\lesssim
  2^{-\frac{k(1-\eta)}2}t^{-\frac{1+\eta}2},\quad\forall\,r>0,
\end{equation}
which is also a consequence of
\begin{equation}\label{equ:highconseq}
  \Big|\int_0^\infty e^{i2^kt(\sqrt{2^{-2k}+\lambda^2}\pm \frac{\lambda r}{t})}  \phi(\lambda)\lambda a_\pm(2^k\lambda r) \;d\lambda\Big|\lesssim
  2^{-\frac{k}{2}}t^{-\frac12}(1+2^{-k}t)^{-\frac12}.
\end{equation}
Now, we turn to prove \eqref{equ:highconseq}. Denote
$$I_k(t,x):=\int_0^\infty e^{i2^kt\Phi_k(\lambda,t,r)}  \phi(\lambda)\lambda a_\pm(2^k\lambda r) \;d\lambda,$$
with $\Phi_{k,\pm}(\lambda,t,r)=\sqrt{2^{-2k}+\lambda^2}\pm \frac{\lambda r}{t}.$
Note that
$$\partial_\lambda\Phi_{k,\pm}=\frac{\lambda}{\sqrt{2^{-2k}+\lambda^2}}\pm\frac{r}{t},$$
for $r\geq 100t$ or $r\leq\frac{t}{100}$ we get
$$|\partial_\lambda\Phi_{k,\pm}|\geq\frac1{100}\quad\text{and}\quad \partial_\lambda^2\Phi_{k,\pm}=\frac{2^{-2k}}{(2^{-2k}+\lambda^2)^{\frac32}}>0.$$
In this case, by Van-der Corput Lemma and \eqref{equ:akrder}, we get
\begin{align*}
  \big|\partial_\lambda\Phi_{k,\pm}\big|\lesssim & (2^kt)^{-1}\int_\frac12^2\big|\pa_\lambda\big[ \phi(\lambda)\lambda a_\pm(2^k\lambda r) \big]\big|\;d\lambda \lesssim2^{-\frac{k}{2}}t^{-\frac12}(1+2^{-k}t)^{-\frac12}.
\end{align*}
We finally consider the case that  $\frac{t}{100}\leq r\leq 100t.$   It is easy to check that
$$|\partial_\lambda^2\Phi_{k,\pm}|=\frac{2^{-2k}}{(2^{-2k}+\lambda^2)^\frac32}\gtrsim 2^{-2k}.$$
Hence, if $t2^{-k}\geq 1$, by Van-der Corput Lemma and \eqref{equ:akrder}, we obtain
\begin{align*}
  |I_k(t,r)|\lesssim  & (2^{-k}t)^{-\frac12}\int_{\frac12}^2\big|\pa_\lambda\big[\phi(\lambda)\lambda a_\pm(2^k\lambda r)  \big]\big|\,d\lambda\\
  \lesssim&2^{\frac k2}t^{-\frac12}(2^kr)^{-\frac12}\lesssim 2^{-\frac{k}{2}}t^{-\frac12}(1+2^{-k}t)^{-\frac12}.
  \end{align*}
Otherwise $t2^{-k}\leq 1$, we directly have
\begin{align*}
  |I_k(t,r)|\lesssim  & \int_{\frac12}^2\big|\big[\phi(\lambda)\lambda a_\pm(2^k\lambda r)  \big]\big|\,d\lambda
  \lesssim (2^kr)^{-\frac12}\lesssim 2^{-\frac{k}{2}}t^{-\frac12}(1+2^{-k}t)^{-\frac12}.
  \end{align*}
Thus, we complete the proof of Proposition \ref{dispersive-h}.
\end{proof}

As consequences of Proposition \ref{dispersive-l} and Proposition \ref{dispersive-h} respectively, we
immediately have
\begin{proposition}\label{prop:Dispersive-l} Let $U^{\mathrm{low}}(t)$ be defined by
\begin{equation}\label{equ:lowparkg}
 U^{\mathrm{low}}(t):=\int_0^\infty e^{it\sqrt{1+\lambda^2}} \varphi_0(\lambda) dE_{\sqrt{\LL_{{\A},0}}}(\lambda).
\end{equation}
Then there exists a constant $C$ independent of $t, x, y$  such that
\begin{equation}\label{Dispersive}
\|U^\mathrm{low}(t)(U^\mathrm{low})^*(s)\|_{L^1\rightarrow L^\infty}\leq C
(1+|t-s|)^{-1}.
\end{equation}
\end{proposition}

\begin{proposition}\label{prop:Dispersive-h} Let $U_{k}(t)$ be defined by
\begin{equation}\label{U-k}
  U_k(t):=\int_0^\infty e^{it\sqrt{1+\lambda^2}} \phi(2^{-k}\lambda) dE_{\sqrt{\LL_{{\A},0}}}(\lambda),\quad k\in\Z.
\end{equation}
Then there exists a constant $C$ independent of $t, x, y$ for all
$k\geq 1$ such that
\begin{equation}\label{Dispersive123}
\|U_{k}(t)U^*_{k}(s)\|_{L^1\rightarrow L^\infty}\leq C
2^{k\frac{3+\eta}2}(2^{-k}+|t-s|)^{-\frac{1+\eta}2}
\end{equation}
where $0\leq \eta\leq 1$.
\end{proposition}

\section{Strichartz estimates for purely magnetic Klein-Gordon}

In this section, we show the Strichartz estimates in Theorem
\ref{thm:Stri} with $a\equiv0$. The argument is standard once one proved the
dispersive estimates. For convenience, we sketch the main steps, for example, see \cite{KT, ZZ2, ZZ19}.

\subsection{Semiclassical type Strichartz estimates}
We need a variety of the abstract Keel-Tao's Strichartz estimates
theorem. This is an analogue of the semiclassical Strichartz
estimates for Schr\"odinger in \cite{KTZ, Zworski}, see also in \cite{ZZ19}.

\begin{proposition}\label{prop:semi}
Let $(X,\mathcal{M},\mu)$ be a $\sigma$-finite measured space and
$U: \mathbb{R}\rightarrow B(L^2(X,\mathcal{M},\mu))$ be a weakly
measurable map satisfying, for some constants $C$, $\alpha\geq0$,
$\sigma, h>0$,
\begin{equation}\label{md}
\begin{split}
\|U(t)\|_{L^2\rightarrow L^2}&\leq C,\quad t\in \mathbb{R},\\
\|U(t)U(s)^*f\|_{L^\infty}&\leq
Ch^{-\alpha}(h+|t-s|)^{-\sigma}\|f\|_{L^1}.
\end{split}
\end{equation}
Then for every pair $q,p\in[1,\infty]$ such that $(q,p,\sigma)\neq
(2,\infty,1)$ and
\begin{equation*}
\frac{1}{q}+\frac{\sigma}{p}\leq\frac\sigma 2,\quad q\ge2,
\end{equation*}
there exists a constant $\tilde{C}$ only depending on $C$, $\sigma$,
$q$ and $r$ such that
\begin{equation}\label{equ:gstri}
\Big(\int_{\R}\|U(t) u_0\|_{L^p}^q dt\Big)^{\frac1q}\leq \tilde{C}
\Lambda(h)\|u_0\|_{L^2}
\end{equation}
where $\Lambda(h)=h^{-(\alpha+\sigma)(\frac12-\frac1p)+\frac1q}$.
\end{proposition}

\subsection{Homogeneous Strichartz estimates} By the Littlewood-Paley theory associated with the operator $\LL_{{\A},0}$ established in \cite[Proposition 3.1]{FZZ},  the homogeneous Strichartz estimates are
reduced to frequency localized estimates.

Let $U(t)=e^{it\sqrt{1+\LL_{{\A},0}}}$, then
we can write
\begin{equation}\label{sleq}
\begin{split}
u(t,x)
=\frac{U(t)+U(-t)}2 u_0+\frac{U(t)-U(-t)}{2i\sqrt{1+\LL_{{\A},0}}} u_1.
\end{split}
\end{equation}
Without loss of generality, we assume $u_1=0$.
 Notice that
\begin{equation*}
U(t)u_0=\sum_{j\in\Z}\sum_{k\in\Z}U_{k}(t)f_j=U^{\mathrm{low}}(t)u_0+\sum_{j\in\Z}\sum_{k\geq1}U_{k}(t)f_j,
\end{equation*}
where $f_j=\phi\big(2^{-j}\sqrt{\mathcal{L}_{\A,0}}\big)u_0$, $U^{\mathrm{low}}(t)$ is given in \eqref{equ:lowparkg} and $U_k(t)$ is in \eqref{U-k}.
By the Littlewood-Paley square inequality in \cite[Proposition 3.1]{FZZ} and Minkowski's inequality,  for
$q,p\geq2$, we have
\begin{equation}\label{LP}
\begin{split}
&\|U(t)u_0\|_{L^q(\R;L^p(\R^2))}\\ \lesssim& \|U^{\mathrm{low}}(t)u_0\|_{L^q(\R;L^p(\R^2))}+
\Big(\sum_{j\in\Z}\|\sum_{k\geq 1}U_{k}(t)f_j\|^2_{L^q(\R;L^p(\R^2))}\Big)^{\frac12}.
\end{split}
\end{equation}
On one hand, by using energy estimate,
Proposition \ref{prop:Dispersive-l} and the argument in Keel-Tao \cite{KT}, we have for $2/q\leq 2(1/2-1/p)$
\begin{equation}
\begin{split}
\|U^{\mathrm{low}} u_0\|_{L^q(\R;L^p(\R^2))}\leq C\|u_0\|_{L^2(\R^2)}.
\end{split}
\end{equation}
On the other hand, by using energy estimate and
Proposition \ref{prop:Dispersive-h}, we have the estimates \eqref{md}
for $U_{k}(t)$, where $\alpha=(3+\eta)/2$, $\sigma=(1+\eta)/2$ and
$h=2^{-k}$. Then it follows from Proposition \ref{prop:semi} that
\begin{equation*}
\|U_{k}(t)f_j\|_{L^q_t(\R:L^p(\R^2))}\lesssim
2^{k[(2+\eta)(\frac12-\frac1p)-\frac1q]} \|f_j\|_{L^2(\R^2)}.
\end{equation*}
 In
view of  $f_j=\phi(2^{-j}\sqrt{\LL_{{\A},0}})u_0$, then
$\phi(2^{-k}\sqrt{\LL_{{\A},0}})f_j$ vanishes if
$|j-k|\geq3$.
Then
\begin{equation}
\begin{split}
\sum_{j\in\Z}\|\sum_{k\geq 1}U_{k}(t)f_j\|^2_{L^q(\R;L^p(\R^2))}&\leq C \sum_{j\in\Z} \sum_{|k-j|\leq 3}\|U_{k}(t)f_j\|^2_{L^q(\R;L^p(\R^2))}\\
 &\leq C\sum_{j\geq0}2^{2js}\|f_j\|^2_{L^2(\R^2)}\leq C\|u_0\|^2_{ H^{s}_{{\A},0}(\R^2)}.
\end{split}
\end{equation}
Therefore we prove the Strichartz estimate with $u_1=F=0$
\begin{equation}
\|u\|_{L^q(\R;L^p(\R^2))}\leq C\|u_0\|_{ H^{s}_{{\A},0}(\R^2)}.
\end{equation}
holds for all $(q,p,s)$ satisfying \eqref{adm} and \eqref{scaling}.

\subsection{Inhomogeneous Strichartz estimates}\label{sec:inh} Let
$U(t)=e^{it\sqrt{1+\LL_{{\A},0}}}: L^2\rightarrow L^2$, then we have already
proved that
\begin{equation}
\|U(t)u_0\|_{L^q_tL^p_x}\lesssim\|u_0\|_{H^{s}_{{\A},0}(\R^2)}
\end{equation} holds for all $(q,p)\in \Lambda_{s,\eta}$ where $s\in\R$ and $0\leq\eta\leq 1$.
For $s\in\R$ and $(q,p)\in \Lambda_{s,\eta}$,
we define the operator $T_s$ by
\begin{equation}\label{Ts}
\begin{split}
T_s: L^2_x(\R^2)&\rightarrow L^q_tL^p_x(\R\times\R^2),\quad f\mapsto (1+\LL_{{\A},0})^{-\frac
s2}e^{it\sqrt{1+\LL_{{\A},0}}}f.
\end{split}
\end{equation}
Then we have by duality
\begin{equation}\label{Ts*}
\begin{split}
T^*_{1-s}: L^{\tilde{q}'}_tL^{\tilde{p}'}_x(\R\times\R^2)\rightarrow L^2(\R^2),\quad
F(\tau,x)&\mapsto \int_{\R}(1+\LL_{{\A},0})^{\frac
{s-1}2}e^{-i\tau\sqrt{1+\LL_{{\A},0}}}F(\tau)d\tau,
\end{split}
\end{equation}
where $1-s=2(\frac12-\frac1{\tilde{p}})-\frac1{\tilde{q}}$.
Therefore we obtain
\begin{equation*}
\Big\|\int_{\R}U(t)U^*(\tau)(1+\LL_{{\A},0})^{-\frac12}F(\tau)d\tau\Big\|_{L^q_tL^p_x(\R\times\R^2)}
=\big\|T_sT^*_{1-s}F\big\|_{L^q_tL^p_x(\R\times\R^2)}\lesssim\|F\|_{L^{\tilde{q}'}_tL^{\tilde{p}'}_x(\R\times\R^2)}.
\end{equation*}
Since $s=(2+\eta)(\frac12-\frac1p)-\frac1q$ and
$1-s=(2+\eta)(\frac12-\frac1{\tilde{p}})-\frac1{\tilde{q}}$, thus $(q,p),
(\tilde{q},\tilde{p})$ satisfy \eqref{scaling}. By the
Christ-Kiselev lemma \cite{CK}, we thus obtain for $q>\tilde{q}'$,
\begin{equation}\label{non-inhomgeneous}
\begin{split}
\Big\|\int_{\tau<t}\frac{\sin{(t-\tau)\sqrt{1+\LL_{{\A},0}}}}
{\sqrt{1+\LL_{{\A},0}}}F(\tau)d\tau\Big\|_{L^q_tL^p_x(\R\times\R^2)}\lesssim\|F\|_{L^{\tilde{q}'}_t{L}^{\tilde{p}'}_x(\R\times\R^2)}.
\end{split}
\end{equation}
Notice that for all $(q,p), (\tilde{q},\tilde{p})$ satisfy
\eqref{adm} and \eqref{scaling}, we must have $q>\tilde{q}'$.
Therefore we have proved all inhomogeneous Strichartz estimates
including $q=2$.

In sum, we have

\begin{theorem}[Global-in-time Strichartz estimate]\label{thm:Stri-0}   Let $a\equiv0$, the Strichartz estimates \eqref{stri} hold
for $(q,p), (\tilde{q},\tilde{p})\in \Lambda_{s,\eta}$
 with $0\leq\eta\leq1$ and $s\in\R$.

\end{theorem}

\section{The heat kernel estimate and Sobolev embedding}

To prove Theorem
\ref{thm:Stri}, we need to take account into the perturbation both from the magnetic potential $\A$ and electrical potential $a$.
In this section, as preliminary tools, we prove Proposition \ref{prop:heat} and then show the Sobolev embedding associated with the operator $\LL_{{\A},a}$.
\begin{proposition}[Sobolev inequality]\label{prop:sob} Let $\LL_{{\A},a}$ be as in Proposition \ref{prop:heat}. Then there exists a constant $C$ such that
\begin{equation}\label{est:sobolev}
\|f\|_{L^q(\R^2)}\leq C\|\LL_{{\A},a}^{\frac{\sigma}2}f\|_{L^p(\R^2)},
\end{equation}
where $0\leq \sigma=2(\frac1p-\frac1q)<2$ and $1<q,p<\infty$.
\end{proposition}
\vspace{0.2cm}

\begin{proof} To prove \eqref{est:sobolev}, it suffices to show
\begin{equation}\label{est:sobolev'}
\|\LL_{{\A},a}^{-\frac{\sigma}2} f\|_{L^q(\R^2)}\leq C\|f\|_{L^p(\R^2)},
\end{equation}
where
\begin{equation}
\LL_{{\A},a}^{-\frac{\sigma}2}:=\frac{1}{\Gamma(\sigma/2)}\int_0^\infty e^{-t\LL_{{\A},a}} t^{\frac \sigma2} \frac{dt}{t},\quad \sigma\geq0.
\end{equation}
By Proposition \ref{prop:heat}, we have
\begin{equation}\label{est:sobolev''}
\begin{split}
\|\LL_{{\A},a}^{-\frac{\sigma}2} f\|_{L^q(\R^2)}&\leq C_{\sigma}  \Big\|\int_{\R^2}\int_0^\infty e^{-\frac{|x-y|^2}{ct}} t^{\frac\sigma2-2} dt f(y) dy \Big\|_{L^q(\R^2)}\\
&\leq C_{\sigma} \Big\|\int_{\R^2} |x-y|^{\sigma-2} f(y) dy \Big\|_{L^q(\R^2)} \int_0^\infty e^{-\frac{1}{ct}} t^{\frac\sigma2-2} dt \\
&\leq C_{\sigma} \|f\|_{L^p(\R^2)},
\end{split}
\end{equation}
where we have used the Hardy-Littlewood-Sobolev inequality in the last inequality.
\end{proof}

Thus our main task is to prove Proposition \ref{prop:heat}. Our strategy is to combine the asymptotic property of the spectrum of $L_{{\A},a}$  proved in \cite{FFFP} and the construction of heat kernel in \cite{FZZ} when $a\equiv0$.
Recall
\begin{equation}\label{equ:laadefadd}
\begin{split}
L_{{\A},a}&=-\Delta_{\mathbb{S}^{1}}+\big(|{\A}(\hat{x})|^2+a(\hat{x})+i\,\mathrm{div}_{\mathbb{S}^{1}}{\A}(\hat{x})\big)+2i {\A}(\hat{x})\cdot\nabla_{\mathbb{S}^{1}}\\
&=-\partial_\theta^2+\big(|{A}(\theta)|^2+a(\theta)+i\,{A'}(\theta)\big)+2i {A}(\theta)\partial_\theta.
\end{split}
\end{equation}

From the classical spectral theory, the spectrum of $L_{{\A},a}$ is formed by a countable family of real eigenvalues with finite multiplicity $\{\mu_k({\A},a)\}_{k=1}^\infty$ enumerated  such that
\begin{equation}\label{eig-Aa}
\mu_1({\A},a)\leq \mu_2({\A},a)\leq \cdots
\end{equation}
where we repeat each eigenvalue as many times as its multiplicity, and $\lim\limits_{k\to\infty}\mu_k({\A},a)=+\infty$, see \cite[Lemma A.5]{FFT}.

For each $k\in\N, k\geq1$, let $\psi_k(\hat{x})\in L^2(\mathbb{S}^{1})$ be the normalized eigenfunction of the operator $L_{{\A},a}$ corresponding to the $k$-th eigenvalue $\mu_k({\A},a)$, i.e. satisfying that
\begin{equation}\label{equ:eig-Aa}
\begin{cases}
L_{{\A},a}\psi_k(\hat{x})=\mu_k({\A},a)\psi_k(\hat{x})\quad \text{on}\,\quad  \mathbb{S}^{1};\\
\int_{\mathbb{S}^{1}}|\psi_k(\hat{x})|^2 d\hat{x}=1.
\end{cases}\end{equation}
To prove Proposition \ref{prop:heat}, we consider two cases, i.e., non-resonant case $\Phi_{{\A}}\not \in\frac12\Z$ and resonant case $\Phi_{{\A}} \in\frac12\Z$.\vspace{0.2cm}

{\bf Case I: non-resonant case $\Phi_{{\A}}\not \in\frac12\Z$. }
We recall \cite[Lemma 2.1]{FFFP} about the asymptotic property about the spectrum.
To consist with the notation of \cite{FFFP}, we set $\tilde{A}=\Phi_{{\A}}$ in this subsection.
\begin{lemma}\label{lem:asy-eig}  Let $a(\theta)\in W^{1,\infty}(\mathbb{S}^1)$ and $A(\theta)\in W^{1,\infty}(\mathbb{S}^1)$ and set
\begin{equation}
 \tilde{A}=\frac1{2\pi}\int_0^{2\pi} A(\theta') d\theta'\notin \frac12\Z,\qquad \tilde{a}=\frac1{2\pi}\int_0^{2\pi} a(\theta') d\theta'.
\end{equation}
Let $\mu_k=\mu_k({\A},a)$ be given in \eqref{eig-Aa}. Then there exist $K, J\in\N$ large enough such that $\{\mu_k, k\geq K\}=\{\nu_j, j\in\Z, |j|\geq J\}$ such that
\begin{equation}
\sqrt{\nu_j-\tilde{a}}=(\mathrm{sign} j)\big(\tilde{A}-\big\lfloor\tilde{A}+\frac12\big\rfloor\big)+|j|+O(|j|^{-3}),\quad |j| \to\infty,
\end{equation}
and
\begin{equation}\label{njA}
\nu_j=\tilde{a}+\big(j+\tilde{A}-\big\lfloor\tilde{A}+\frac12\big\rfloor\big)^2+O(|j|^{-2}),\quad |j| \to\infty.
\end{equation}
Moreover, for all $j\in\Z$ such that $|j|\geq J$, there exists $\phi_j(\hat{x})\in L^2(\mathbb{S}^{1})$ be the normalized eigenfunction of the operator $L_{{\A},a}$ corresponding to the $j$-th eigenvalue $\nu_j({\A},a)$, i.e. satisfying that
\begin{equation}\label{eig-Aa'}
\begin{cases}
L_{{\A},a}\phi_j(\hat{x})=\nu_j({\A},a)\phi_j(\hat{x})\quad \text{on}\,\quad  \mathbb{S}^{1};\\
\int_{\mathbb{S}^{1}}|\phi_j(\hat{x})|^2 d\hat{x}=1.
\end{cases}\end{equation}
In addition, let $\hat{x}=(\cos\theta,\sin\theta)$ with $\theta\in [0,2\pi]$, then
\begin{equation}\label{eigen'}
\begin{split}
\phi_j(\hat{x})=\phi_j(\theta)=\frac1{\sqrt{2\pi}}e^{-i\big(\lfloor\tilde{A}+1/2\rfloor\theta+\int_0^\theta A(\theta')d\theta'\big)}\Big(e^{i(\tilde{A}+j)\theta}+R_j(\theta)\Big),
\end{split}\end{equation}
where $\|R_j(\theta)\|_{L^\infty(\mathbb{S}^1)}=O(|j|^{-3})$ as $|j|\to+\infty$. The notation $\lfloor\cdot\rfloor$ denotes the floor function $\lfloor x\rfloor=\max\{k\in\Z: k\leq x\}$.
\end{lemma}

Now we prove the heat kernel estimate in Proposition \ref{prop:heat} .

\begin{proof}[{\bf The proof of Proposition \ref{prop:heat}}]

By using the functional calculus (see \cite{Taylor}), we write the heat kernel
\begin{equation}\label{heat-ke}
e^{-t\LL_{{\A},a}}(x,y)=K(t; r_1,\theta_1,r_2,\theta_2)=\sum_{k=1}^\infty \psi_{k}(\theta_1) \overline{\psi_{k}(\theta_2)}K_{\sqrt{\mu_k}}(t; r_1,r_2).
\end{equation}
where
\begin{equation}\label{equ:knukdef1new}
  K_{\sqrt{\mu_k}}(t; r_1,r_2)=\int_0^\infty e^{-t\rho^2} J_{\sqrt{\mu_k}}(r_1\rho)J_{\sqrt{\mu_k}}(r_2\rho) \,\rho d\rho=t^{-1}e^{-\frac{r_1^2+r_2^2}{4t}} I_{\sqrt{\mu_k}}\big(\frac{r_1r_2}{2t}\big).
\end{equation}
and $J_{\sqrt{\mu_k}}$ is the Bessel function of order $\sqrt{\mu_k}$ and $I_{\sqrt{\mu_k}}$ is the modified Bessel function of the first kind.
 Recall
\begin{equation}
I_{\sqrt{\mu_k}}(z)=\frac1{\pi}\int_0^\pi e^{z\cos(s)} \cos(\sqrt{\mu_k} s) ds-\frac{\sin(\sqrt{\mu_k}\pi)}{\pi}\int_0^\infty e^{-z\cosh s} e^{-s\sqrt{\mu_k}} ds,
\end{equation}
let $z=\frac{r_1r_2}{2t}$, then we write the heat kernel as
\begin{align}\label{heat-ke'}
e^{-t\LL_{{\A},a}}(x,y)=&t^{-1}e^{-\frac{r_1^2+r_2^2}{4t}}\sum_{k=1}^\infty \psi_{k}(\theta_1) \overline{\psi_{k}(\theta_2)} I_{\sqrt{\mu_k}}\big(z\big)\\\nonumber
=&(\pi t)^{-1}e^{-\frac{r_1^2+r_2^2}{4t}}\sum_{k=1}^\infty \psi_{k}(\theta_1) \overline{\psi_{k}(\theta_2)} \int_0^\pi e^{z\cos(s)} \cos(\sqrt{\mu_k} s) ds\\\nonumber
&-(\pi t)^{-1}e^{-\frac{r_1^2+r_2^2}{4t}}\sum_{k=1}^\infty \psi_{k}(\theta_1) \overline{\psi_{k}(\theta_2)}\sin(\sqrt{\mu_k}\pi)\int_0^\infty e^{-z\cosh s} e^{-s\sqrt{\mu_k}} ds\\\nonumber
:=&G(t;x,y)+D(t;x,y).
\end{align}

{\bf Step 1:} We first consider the term $G(t;x,y)$. Since
\begin{equation}\label{equL:heat-ke'}
\begin{split}
\cos(s\sqrt{L_{{\A},a}} )=\sum_{k=1}^\infty \psi_{k}(\theta_1) \overline{\psi_{k}(\theta_2)}  \cos(\sqrt{\mu_k} s)
\end{split}
\end{equation}
where $L_{{\A},a}$ is the operator on $\mathbb{S}^1$ and $a(\theta)\in L^1_{\text{loc}}$. By \cite[Theorem 3.3]{CS}, the propagator $\cos(s\sqrt{L_{{\A},a}})$ has finite
propagation speed, thus $\cos(s\sqrt{L_{{\A},a}})$ vanishes $s<|\theta_1-\theta_2|$.
Then
\begin{equation}
\begin{split}
G(t;x,y)=(\pi t)^{-1}e^{-\frac{r_1^2+r_2^2}{4t}}\sum_{k=1}^\infty \psi_{k}(\theta_1) \overline{\psi_{k}(\theta_2)} \int_{|\theta_1-\theta_2|}^\pi e^{z\cos(s)} \cos(\sqrt{\mu_k} s) ds.
\end{split}
\end{equation}
For $K, J$ in Lemma \ref{lem:asy-eig}, we have
$$\{\mu_k, k\geq K\}=\{\nu_j, j\in\Z, |j|\geq J\}.$$
Hence, similarly as \cite[Corollary 2.12]{FFT}, for any $ k\geq K$, there exists $|j|\geq J$ such that $\mu_k=\nu_j$ and $\psi_{k}(\theta)=\phi_j(\theta)$.
We split $G(t;x,y)$ into two terms
$$G(t;x,y)=G_1(t;x,y)+G_2(t;x,y)$$
where
\begin{equation}\label{G1}
G_1(t;x,y)=(\pi t)^{-1}e^{-\frac{r_1^2+r_2^2}{4t}}\sum_{k=1}^{K} \psi_{k}(\theta_1) \overline{\psi_{k}(\theta_2)} \int_{|\theta_1-\theta_2|}^\pi e^{z\cos(s)} \cos(\sqrt{\mu_k} s) ds,
\end{equation}
and
\begin{align}\label{G2}
&G_2(t;x,y)=(\pi t)^{-1}e^{-\frac{r_1^2+r_2^2}{4t}}e^{-i\big(\lfloor\tilde{A}+1/2\rfloor(\theta_1-\theta_2)+\int_{\theta_1}^{\theta_2} A(\theta')d\theta'\big)} \\\nonumber
&\times \sum_{\{ j\in\Z, |j|\geq J\}}  \Big(e^{i(\tilde{A}+j)\theta_1}+R_j(\theta_1)\Big)
\Big(e^{-i(\tilde{A}+j)\theta_2}+\overline{R_j(\theta_2)}\Big) \int_{|\theta_1-\theta_2|}^\pi e^{z\cos(s)} \cos(\sqrt{\nu_j} s) ds.
\end{align}
We estimate $G_1(t;,x,y)$ by
\begin{equation}\label{G1'}
|G_1(t;x,y)|\leq C_K t^{-1}e^{-\frac{r_1^2+r_2^2}{4t}}e^{z\cos(\theta_1-\theta_2)} \leq C_Kt^{-1} e^{-\frac{|x-y|}{4t}},
\end{equation}
where we use the fact $z=\frac{r_1r_2}{2t}$ and $|\psi_k|\leq C$. For $G_2(t;,x,y)$, since $|R_j|\leq C_J |j|^{-3}$, then
\begin{align}\nonumber
|G_2(t;x,y)|\leq& C_J t^{-1}e^{-\frac{r_1^2+r_2^2}{4t}}e^{z\cos(\theta_1-\theta_2)}\sum_{|j|\geq J} |j|^{-3}\\\nonumber
&+C t^{-1}e^{-\frac{r_1^2+r_2^2}{4t}}
\Big| \sum_{\{ j\in\Z, |j|\geq J\}}  e^{i(\tilde{A}+j)(\theta_1-\theta_2)}\int_{|\theta_1-\theta_2|}^\pi e^{z\cos(s)} \cos(\sqrt{\nu_j} s) ds\Big|\\\label{G2'}
\leq& C_Jt^{-1} e^{-\frac{|x-y|}{4t}},
\end{align}
if we could prove that
\begin{lemma}\label{lem:G2-key} There holds
\begin{equation}\label{G2-key}
\begin{split}
\Big| \sum_{\{ j\in\Z, |j|\geq J\}}  e^{ij(\theta_1-\theta_2)}\int_{|\theta_1-\theta_2|}^\pi e^{z\cos(s)} \cos(\sqrt{\nu_j} s) ds\Big|
\leq C_J e^{z\cos(\theta_1-\theta_2)}.
\end{split}
\end{equation}
\end{lemma}

\begin{proof}
Let $j_A=j+\bar{A}$ with $\bar{A}=\tilde{A}-\big\lfloor\tilde{A}+\frac12\big\rfloor$, we write
  \begin{equation}
\begin{split}
&\sum_{\{ j\in\Z, |j|\geq J\}} \cos(\sqrt{\nu_j} s)  e^{ij(\theta_1-\theta_2)}\\
=&\sum_{\{ j\in\Z, |j|\geq J\}} \big(\cos(\sqrt{\nu_j} s)-\cos(|j_A|s)\big)  e^{ij(\theta_1-\theta_2)}\\
&-\sum_{\{ j\in\Z, |j|\leq J\}} \cos(|j_A|s) e^{ij(\theta_1-\theta_2)}+\sum_{ j\in\Z} \cos(|j_A|s) e^{ij(\theta_1-\theta_2)}.
  \end{split}
  \end{equation}

We first prove
\begin{equation}\label{G2-key1}
\begin{split}
&\Big| \int_{|\theta_1-\theta_2|}^\pi e^{z\cos(s)} \sum_{\{ j\in\Z, |j|\leq J\}} \cos(|j_A|s) e^{ij(\theta_1-\theta_2)} ds\Big|\\
\leq& C_J \int_{|\theta_1-\theta_2|}^\pi e^{z\cos(s)} ds
\leq C_J e^{z\cos(\theta_1-\theta_2)}.
\end{split}
\end{equation}
Next, similarly as \cite{FZZ}, we consider
\begin{equation}\label{G2-key2}
\begin{split}
\int_{|\theta_1-\theta_2|}^\pi e^{z\cos(s)} \sum_{ j\in\Z} \cos(|j_A|s) e^{ij(\theta_1-\theta_2)} ds.
\end{split}
\end{equation}
By the Poisson summation formula, we write
\begin{equation*}
\begin{split}
\sum_{ j\in\Z} \cos(|j_A|s) e^{ij(\theta_1-\theta_2)}&=\sum_{j\in\Z} \frac{e^{i(j+\bar{A})s}+e^{-i(j+\bar{A})s}}2 e^{ij(\theta_1-\theta_2)}\\
&=\frac12\sum_{j\in\Z}\big(e^{i\bar{A} s}\delta(\theta_1-\theta_2+s+2\pi j)+e^{-i\bar{A} s}\delta(\theta_1-\theta_2-s+2\pi j)\big).
\end{split}
\end{equation*}
Therefore we have
\begin{equation}\label{equL:G2-key2}
\begin{split}
&\Big| \int_{|\theta_1-\theta_2|}^\pi e^{z\cos(s)} \sum_{ j\in\Z} \cos(|j_A|s) e^{ij(\theta_1-\theta_2)} ds\Big|\\
\leq& C_J \sum_{j\in\Z}\int_{|\theta_1-\theta_2|}^\pi e^{z\cos(s)} \delta(\theta_1-\theta_2\pm s+2\pi j)ds
\leq C_J e^{z\cos(\theta_1-\theta_2)},
\end{split}
\end{equation}
where we have used the fact that $0\leq s=\pm (\theta_1-\theta_2+2\pi j)\leq 2\pi$ implies the summation in $j$ is finite.\vspace{0.2cm}

Finally, by Taylor expanding $\cos\alpha-\cos\beta=-\sin\beta (\alpha-\beta)+O((|\alpha-\beta|)^2)$, we aim to estimate
\begin{equation}\label{G2-key3}
\begin{split}
&\Big| \int_{|\theta_1-\theta_2|}^\pi e^{z\cos(s)} \sum_{\{ j\in\Z, |j|\geq J\}} \big(\cos(\sqrt{\nu_j} s)-\cos(|j_A|s)\big)  e^{ij(\theta_1-\theta_2)} ds\Big|\\
\leq& \Big| \int_{|\theta_1-\theta_2|}^\pi e^{z\cos(s)} \sum_{\{ j\in\Z, |j|\geq J\}} \sin(|j_A|s) (\sqrt{\nu_j}-|j_A|)s  e^{ij(\theta_1-\theta_2)} ds\Big|\\
&+\sum_{\{ j\in\Z, |j|\geq J\}} (\sqrt{\nu_j}-|j_A|)^2 \Big| \int_{|\theta_1-\theta_2|}^\pi e^{z\cos(s)}  s^2 ds\Big|.
\end{split}
\end{equation}
From \eqref{njA}, notice that
$$\sqrt{\nu_j}-|j_A|=\frac{\nu_j-j_A^2}{\sqrt{\nu_j}+j_A}=\frac{\tilde{a}}{|j|}+O(|j|^{-2}),$$
we therefore obtain
\begin{equation}
\begin{split}
&\sum_{\{ j\in\Z, |j|\geq J\}} (\sqrt{\nu_j}-|j_A|)^2 \Big| \int_{|\theta_1-\theta_2|}^\pi e^{z\cos(s)}  s^2 ds\Big|\\&\leq C \sum_{\{ j\in\Z, |j|\geq J\}} |j|^{-2} e^{z\cos(\theta_1-\theta_2)}
\leq C_J e^{z\cos(\theta_1-\theta_2)},
\end{split}
\end{equation}
and
\begin{equation}\label{G2-key3'}
\begin{split}
& \Big| \int_{|\theta_1-\theta_2|}^\pi e^{z\cos(s)} \sum_{\{ j\in\Z, |j|\geq J\}} \sin(|j_A|s) (\sqrt{\nu_j}-|j_A|)s  e^{ij(\theta_1-\theta_2)} ds\Big|\\
\leq&  \Big| \int_{|\theta_1-\theta_2|}^\pi e^{z\cos(s)} \sum_{\{ j\in\Z, |j|\geq J\}} \sin(|j_A|s) \frac{s}{|j|}  e^{ij(\theta_1-\theta_2)} ds\Big|+C_J e^{z\cos(\theta_1-\theta_2)}\\
\leq&  \Big| \int_{|\theta_1-\theta_2|}^\pi e^{z\cos(s)} \sum_{\{ j\in\Z, |j|\geq 1\}} \frac{e^{is|j_A|}-e^{-is|j_A|}}{2i} \frac{s}{|j|}  e^{ij(\theta_1-\theta_2)} ds\Big|+C_J e^{z\cos(\theta_1-\theta_2)}.
\end{split}
\end{equation}
Recall $j_A=j+\bar{A}$ where $\bar{A}=\tilde{A}-\big\lfloor\tilde{A}+\frac12\big\rfloor$, we note that
\begin{equation}
\begin{split}
 &\sum_{\{ j\in\Z, |j|\geq 1\}} e^{is|j_A|}\frac{1}{|j|}  e^{ij(\theta_1-\theta_2)} \\
 =& \sum_{\{ j\in\Z, j\geq 1\}} \frac{1}{j}  e^{i[j(\theta_1-\theta_2)+(j+\bar{A})s]} - \sum_{\{ j\in\Z, j\leq -1\}} \frac{1}{j}  e^{i[j(\theta_1-\theta_2)-(j+\bar{A})s]} \\
\end{split}
\end{equation}
and
\begin{equation}
\begin{split}
 &\sum_{\{ j\in\Z, |j|\geq 1\}} e^{-is|j_A|}\frac{1}{|j|}  e^{ij(\theta_1-\theta_2)} \\
 =& \sum_{\{ j\in\Z, j\geq 1\}} \frac{1}{j}  e^{i[j(\theta_1-\theta_2)-(j+\bar{A})s]} - \sum_{\{ j\in\Z, j\leq -1\}} \frac{1}{j}  e^{i[j(\theta_1-\theta_2)+(j+\bar{A})s]}.
\end{split}
\end{equation}
Therefore we obtain
\begin{equation}
\begin{split}
 &\sum_{\{ j\in\Z, |j|\geq 1\}} \big(e^{is|j_A|}-e^{-is|j_A|}\big)\frac{1}{|j|}  e^{ij(\theta_1-\theta_2)} \\
 =& \sum_{\{ j\in\Z\setminus\{0\}\}} \frac{1}{j}  e^{i[j(\theta_1-\theta_2)+(j+\bar{A})s]} - \sum_{\{ j\in\Z\setminus\{0\}\}} \frac{1}{j}  e^{i[j(\theta_1-\theta_2)-(j+\bar{A})s]}\\
 =& e^{i\bar{A}s} \sum_{\{ j\in\Z\setminus\{0\}\}} \frac{1}{j}  e^{ij(\theta_1-\theta_2+s)} - e^{-i\bar{A}s} \sum_{\{ j\in\Z\setminus\{0\}\}} \frac{1}{j}  e^{ij(\theta_1-\theta_2-s)}.
\end{split}
\end{equation}
Let $\theta=(\theta_1-\theta_2)\pm s$ with $0\leq s\leq \pi$,
\begin{equation}
\begin{split}
S(\theta)=\sum_{\{ j\in\Z\setminus\{0\}\}} \frac{1}{j}  e^{ij\theta}
\end{split}
\end{equation}
then $S(\theta)$ converges to a periodic function extended by sawtooth function
\begin{equation}
f(\theta)=\begin{cases} \frac{\pi-\theta}2, \quad 0<\theta<2\pi,\\
0,\qquad\quad \theta=0.
\end{cases}
\end{equation}
Plugging this into \eqref{G2-key3}, therefore we prove
\begin{equation}
\begin{split}
\Big| \int_{|\theta_1-\theta_2|}^\pi e^{z\cos(s)} \sum_{\{ j\in\Z, |j|\geq J\}} \sin(|j_A|s) (\sqrt{\nu_j}-|j_A|)s  e^{ij(\theta_1-\theta_2)} ds\Big|\leq C_J e^{z\cos(\theta_1-\theta_2)}.
\end{split}
\end{equation}
\end{proof}

{\bf Step 2:} We next consider the term $D(t;x,y)$. Similarly as arguing $G(t;x,y)$, we use Lemma \ref{lem:asy-eig}
to split $D(t;x,y)$ into two terms
$$D(t;x,y)=D_1(t;x,y)+D_2(t;x,y)$$
where
\begin{equation}\label{D1}
D_1(t;x,y)=(\pi t)^{-1}e^{-\frac{r_1^2+r_2^2}{4t}}\sum_{k=1}^{K} \psi_{k}(\theta_1) \overline{\psi_{k}(\theta_2)} \sin(\sqrt{\mu_k}\pi)\int_0^\infty e^{-z\cosh s} e^{-s\sqrt{\mu_k}} ds,
\end{equation}
and
\begin{align}\label{D2}
&D_2(t;x,y)=(\pi t)^{-1}e^{-\frac{r_1^2+r_2^2}{4t}}e^{-i\big(\lfloor\tilde{A}+1/2\rfloor(\theta_1-\theta_2)+\int_{\theta_1}^{\theta_2} A(\theta')d\theta'\big)} \\\nonumber
&\times \sum_{\{ j\in\Z, |j|\geq J\}}  \Big(e^{i(\tilde{A}+j)\theta_1}+R_j(\theta_1)\Big)
\Big(e^{-i(\tilde{A}+j)\theta_2}+R_j(\theta_2)\Big) \sin(\sqrt{\nu_j}\pi)\int_0^\infty e^{-z\cosh s} e^{-s\sqrt{\nu_j}} ds.
\end{align}
We estimate $D_1(t;,x,y)$ by
\begin{equation}\label{D1'}
|D_1(t;x,y)|\leq C_K t^{-1}e^{-\frac{r_1^2+r_2^2}{4t}}e^{-z}\int_0^\infty e^{-s\sqrt{\mu_1}} ds \leq C_K\mu_1^{-1/2} t^{-1} e^{-\frac{|x-y|}{4t}},
\end{equation}
where we use the fact $z=\frac{r_1r_2}{2t}$ and $|\psi_k|\leq C$. Next we consider $D_2(t;,x,y)$, since $|R_j|\leq C_J |j|^{-3}$ and $\mu_1\leq \nu_j$ when $j\geq J$, then
\begin{align}\nonumber
|D_2(t;x,y)|\leq& C_J t^{-1}e^{-\frac{r_1^2+r_2^2}{4t}}e^{-z}\int_0^\infty e^{-s\sqrt{\mu_1}} ds\sum_{|j|\geq J} |j|^{-3}\\\nonumber
&+C t^{-1}e^{-\frac{r_1^2+r_2^2}{4t}}
\Big| \sum_{\{ j\in\Z, |j|\geq J\}}  e^{i(\tilde{A}+j)(\theta_1-\theta_2)} \sin(\sqrt{\nu_j}\pi)\int_0^\infty e^{-z\cosh s} e^{-s\sqrt{\nu_j}} ds\Big|\\\label{D2'}
\leq& C_Jt^{-1} e^{-\frac{|x-y|}{4t}},
\end{align}
which is a consequence of the following lemma
\begin{lemma}\label{lem:D2-key} There holds
\begin{equation}\label{D2-key}
\begin{split}
\Big| \sum_{\{ j\in\Z, |j|\geq J\}}  e^{ij(\theta_1-\theta_2)} \sin(\sqrt{\nu_j}\pi)\int_0^\infty e^{-z\cosh s} e^{-s\sqrt{\nu_j}} ds\Big|
\leq C_J e^{-z}.
\end{split}
\end{equation}
\end{lemma}

\begin{proof}
We first write
 \begin{equation}
\begin{split}
&\sum_{\{ j\in\Z, |j|\geq J\}}  e^{ij(\theta_1-\theta_2)} \sin(\sqrt{\nu_j} \pi)e^{-s\sqrt{\nu_j}} \\
=&\sum_{\{ j\in\Z, |j|\geq J\}}   e^{ij(\theta_1-\theta_2)}\big(\sin(\sqrt{\nu_j} \pi)-\sin(|j_A| \pi)\big)e^{-s\sqrt{\nu_j}}\\
&+\sum_{\{ j\in\Z, |j|\geq J\}}   e^{ij(\theta_1-\theta_2)}\sin(|j_A| \pi)\big(e^{-s\sqrt{\nu_j}}-e^{-s|j_A|}\big)\\
&+\sum_{\{ j\in\Z, |j|\geq J\}}   e^{ij(\theta_1-\theta_2)}\sin(|j_A| \pi)e^{-s|j_A|}.
  \end{split}
  \end{equation}
 Note that
 $$\sin(\sqrt{\nu_j} \pi)-\sin(|j_A| \pi)=\cos(|j_A|\pi)(\sqrt{\nu_j}-|j_A|)\pi+O(|\sqrt{\nu_j}-|j_A||^2)$$
 and
 $$\sqrt{\nu_j}-|j_A|=\frac{\nu_j-j_A^2}{\sqrt{\nu_j}+j_A}=\frac{\tilde{a}}{|j|}+O(|j|^{-2}),$$
 then
 \begin{equation}\label{equ:ID2-key}
\begin{split}
I:&=\Big| \int_0^\infty e^{-z\cosh s} \sum_{\{ j\in\Z, |j|\geq J\}}   e^{ij(\theta_1-\theta_2)}\big(\sin(\sqrt{\nu_j} \pi)-\sin(|j_A| \pi)\big)e^{-s\sqrt{\nu_j}} ds\Big|\\
&\leq C\int_0^\infty e^{-z\cosh s} \sum_{\{ j\in\Z, |j|\geq J\}}   (|j|^{-1}+O(|j|^{-2}))e^{-s\sqrt{\nu_j}}  ds\\
&\leq C_J e^{-z} \sum_{\{ j\in\Z, |j|\geq J\}}   (|j|^{-1}+O(|j|^{-2}))\nu_j^{-\frac12} \int_0^\infty e^{-s} ds\leq C_J e^{-z}.\\
\end{split}
\end{equation}

 Note that
 $$e^{-s\sqrt{\nu_j}}-e^{-s|j_A|}=e^{-s|j_A|}(\sqrt{\nu_j}-|j_A|)s+O(|\sqrt{\nu_j}-|j_A||^2s^2e^{-cs})$$
 and
 $$\sqrt{\nu_j}-|j_A|=\frac{\nu_j-j_A^2}{\sqrt{\nu_j}+j_A}=\frac{\tilde{a}}{|j|}+O(|j|^{-2}),$$
 then, we obtain
 \begin{equation}\label{equ:IID2-key}
\begin{split}
II:&=\Big| \int_0^\infty e^{-z\cosh s} \sum_{\{ j\in\Z, |j|\geq J\}}   e^{ij(\theta_1-\theta_2)}\sin(|j_A| \pi)\big(e^{-s\sqrt{\nu_j}}-e^{-s|j_A|}\big) ds\Big|\\
&\leq C\int_0^\infty e^{-z\cosh s} \sum_{\{ j\in\Z, |j|\geq J\}}   (|j|^{-1}+O(|j|^{-2}))e^{-s|j_A|} s ds\\
&\leq C_J e^{-z} \sum_{\{ j\in\Z, |j|\geq J\}}   (|j|^{-1}+O(|j|^{-2}))|j_A|^{-2} \int_0^\infty e^{-s} sds\leq C_J e^{-z}.\\
\end{split}
\end{equation}
We follow  the same argument in \cite{FZZ} to obtain
\begin{equation}
\Big| \int_0^\infty e^{-z\cosh s} \sum_{\{ j\in\Z\}}   e^{ij(\theta_1-\theta_2)}\sin(|j_A| \pi)e^{-s|j_A|} ds\Big|\leq Ce^{-z}.
\end{equation}

Recall $|j_A|=|j+\bar{A}|, j\in\Z$ and $\bar{A}\in(-1,1)\setminus\{0\}$,
then
\begin{equation}
|j_A|=|j+\bar{A}|=
\begin{cases}
j+\bar{A},\qquad &j\geq1;\\
|\bar{A}|,\qquad &j=0;\\
-(j+\bar{A}),\qquad &j\leq -1.
\end{cases}
\end{equation}
Therefore we obtain
\begin{equation}
\begin{split}
\sin(\pi|j_A|)=\sin(\pi|j+\bar{A}|)=\begin{cases}
\cos(j\pi)\sin(\bar{A}\pi),\qquad &j\geq1;\\
\sin(|\bar{A}|\pi),\qquad &j=0;\\
-\cos(j\pi)\sin(\bar{A}\pi),\qquad &j\leq -1.
\end{cases}
\end{split}
\end{equation}
Therefore we furthermore have
\begin{equation}
\begin{split}
&\sum_{j\in\Z}\sin(\pi|j+\bar{A}|)e^{-s|j+\bar{A}|}e^{ij(\theta_1-\theta_2)}\\
=&\sin(\bar{A}\pi)\sum_{j\geq 1} \frac{e^{ij\pi}+e^{-ij\pi}}2 e^{-s(j+\bar{A})}e^{ij(\theta_1-\theta_2)}+\sin(|\bar{A}|\pi)e^{-|\bar{A}|s}\\&-\sin(\bar{A}\pi)\sum_{j\leq-1} \frac{e^{ij\pi}+e^{-ij\pi}}2 e^{s(j+\bar{A})}e^{ij(\theta_1-\theta_2)}\\
=&\sin(|\bar{A}|\pi)e^{-|\bar{A}|s}+\frac{\sin(\bar{A}\pi)}2\Big(e^{-s\bar{A}}\sum_{j\geq1} e^{-js} \big(e^{ij(\theta_1-\theta_2+\pi)} +e^{ij(\theta_1-\theta_2-\pi)}\big)\\&\qquad\qquad\qquad- e^{s\bar{A}} \sum_{j\geq 1} e^{-js} \big(e^{-ij(\theta_1-\theta_2+\pi)} +e^{-ij(\theta_1-\theta_2-\pi)}\big)\Big).
\end{split}
\end{equation}
Note that
\begin{equation}\label{sum-j}
\sum_{j=1}^\infty e^{ijz}=\frac{e^{iz}}{1-e^{iz}},\qquad \mathrm{Im} z>0,
\end{equation}
we finally obtain
\begin{align}
&\sum_{j\in\Z}\sin(\pi|j+\bar{A}|)e^{-s|j+\bar{A}|}e^{ij(\theta_1-\theta_2)}\\\nonumber
=&\sin(|\bar{A}|\pi)e^{-|\bar{A}|s}+\frac{\sin(\bar{A}\pi)}2\Big(\frac{e^{-(1+\bar{A})s+i(\theta_1-\theta_2+\pi)}}
{1-e^{-s+i(\theta_1-\theta_2+\pi)}}+\frac{e^{-(1+\bar{A})s+i(\theta_1-\theta_2-\pi)}}{1-e^{-s+i(\theta_1-\theta_2-\pi)}}\\\nonumber
&\qquad\qquad\qquad-\frac{e^{-(1-\bar{A})s-i(\theta_1-\theta_2+\pi)}}{1-e^{-s-i(\theta_1-\theta_2+\pi)}}-
\frac{e^{-(1-\bar{A})s-i(\theta_1-\theta_2-\pi)}}{1-e^{-s-i(\theta_1-\theta_2-\pi)}}\Big)\\\nonumber
=&\sin(|\bar{A}|\pi)e^{-|\bar{A}|s}+\sin(\bar{A}\pi)\Big(\frac{e^{-(1+\bar{A})s+i(\theta_1-\theta_2+\pi)}}
{1-e^{-s+i(\theta_1-\theta_2+\pi)}}-\frac{e^{-(1-\bar{A})s-i(\theta_1-\theta_2+\pi)}}{1-e^{-s-i(\theta_1-\theta_2+\pi)}}\Big)\\\nonumber
=&\sin(|\bar{A}|\pi)e^{-|\bar{A}| s}+\sin(\bar{A}\pi)\frac{(e^{- s}-\cos(\theta_1-\theta_2+\pi))\sinh(\bar{A} s)
+i\sin(\theta_1-\theta_2+\pi)\cosh(\bar{A} s)}{\cosh( s)-\cos(\theta_1-\theta_2+\pi)}.
\end{align}
Therefore it suffices to estimate
\begin{equation*}
\begin{split}
 \int_0^\infty &e^{-z\cosh s} \Big(\sin(|\bar{A}|\pi)e^{-|\bar{A}|s} \\&+\sin(\bar{A}\pi)\frac{(e^{-s}-\cos(\theta_1-\theta_2+\pi))\sinh(\bar{A} s)+i\sin(\theta_1-\theta_2+\pi)\cosh(\bar{A} s)}{\cosh(s)-\cos(\theta_1-\theta_2+\pi)}\Big) ds \leq C e^{-z}.
\end{split}
\end{equation*}
Thus, indeed, we only need to show that for $\bar{A}\in(-1,1)\backslash\{0\}$
  \begin{align}\label{equ:ream1}
    \int_0^\infty e^{-|\bar{A}|s}\;ds\lesssim&1\\\label{equ:ream2}
    \int_0^\infty  \Big|\frac{(e^{-s}-\cos(\theta_1-\theta_2+\pi))\sinh(\bar{A} s)}{\cosh(s)-\cos(\theta_1-\theta_2+\pi)}\Big|\;ds\lesssim&1
    \\\label{equ:ream3}
     \int_0^\infty  \Big|\frac{\sin(\theta_1-\theta_2+\pi)\cosh(\bar{A} s)}{\cosh(s)-\cos(\theta_1-\theta_2+\pi)}\Big|\;ds\lesssim&1.
  \end{align}
  It is easy to check \eqref{equ:ream1}.\vspace{0.2cm}

  {\bf Estimate of \eqref{equ:ream2}:}  Note that
$$\cosh(\tau)
-\cos(\theta_1-\theta_2+\pi)=\sinh^2\big(\tfrac{\tau}2\big)+\sin^2\big(\tfrac{\theta_1-\theta_2+\pi}2\big),$$
  we get
 \begin{align*}
    & \int_0^\infty  \Big|\frac{(e^{-s}-\cos(\theta_1-\theta_2+\pi))\sinh(\bar{A} s)}{\cosh(s)-\cos(\theta_1-\theta_2+\pi)}\Big|\;ds\\
    =&\int_0^1 \Big|\frac{(e^{-s}-1+1-\cos(\theta_1-\theta_2+\pi))\sinh(\bar{A} s)}{\sinh^2\big(\tfrac{s}2\big)+\sin^2\big(\tfrac{\theta_1-\theta_2+\pi}2\big)}\Big|\;ds
     \\&\qquad +\int_1^\infty \Big|\frac{(e^{-s}-\cos(\theta_1-\theta_2+\pi))\sinh(\bar{A} s)}{\sinh^2\big(\tfrac{s}2\big)+\sin^2\big(\tfrac{\theta_1-\theta_2+\pi}2\big)}\Big|\;ds\\
    \lesssim&\int_0^1\frac{(s+\tfrac{\theta_1-\theta_2+\pi}{2})s}{s^2+(\tfrac{\theta_1-\theta_2+\pi}{2})^2}\;ds
    +\int_1^\infty e^{-(1-\bar{A})s}\;ds\\
    \lesssim&1.
 \end{align*}

  {\bf Estimate of \eqref{equ:ream3}:} we have
  \begin{align*}
     &  \int_0^\infty  \Big|\frac{\sin(\theta_1-\theta_2+\pi)\cosh(\bar{A} s)}{\cosh(s)-\cos(\theta_1-\theta_2+\pi)}\Big|\;ds\\
    \lesssim& \int_0^1  \Big|\frac{\sin^2\big(\tfrac{\theta_1-\theta_2+\pi}2\big)}{\sinh^2\big(\tfrac{s}2\big)+\sin^2\big(\tfrac{\theta_1-\theta_2+\pi}2\big)}\Big|\;ds
    +\int_1^\infty \frac{\cosh(\bar{A} s)}{\sinh^2\big(\tfrac{s}2\big)+\sin^2\big(\tfrac{\theta_1-\theta_2+\pi}2\big)}\;ds\\
    \lesssim& \int_0^1  \frac{b}{s^2+b^2}\;ds+\int_1^\infty  e^{-(1-\bar{A})s}\;ds\\
    \lesssim&1
  \end{align*}
  where $b=\sin^2\big(\tfrac{\theta_1-\theta_2+\pi}2\big).$
\end{proof}

{\bf Case II: Resonant case $\Phi_{{\A}}\in\frac12\Z$. } In the resonant case $\Phi_{{\A}}\in\frac12\Z$, they  showed analogous result of
\cite[Lemma 2.1]{FFFP} but with two differences in their preprint version \cite[Lemma B.9 and B.10]{FFFParxiv}. The first one is that, if $a(\theta)$ is non-symmetric,  there exists $\theta_j\in [0,2\pi]$ such that the main part of the eigenfunction is in the form of
$$\cos(j(\theta-\theta_j)),\,\text{when} \; \Phi_{{\A}}\in\Z; \qquad \cos((j+\frac12)(\theta-\theta_j)),\; \text{when}\, \Phi_{{\A}}\in\frac12\Z\setminus\Z.$$
The argument in non-resonant case is too sensitive to recover the issue of $\theta_j$, for example, see the Poisson summation formula and \eqref{sum-j} when $z$ depends on $j$. However, under the assumption that
$a(\theta)$ is symmetric with respect to $\theta=\pi$, this issue disappears since $\theta_j$ vanish.
The second difference is that the remainder term $R_j$ of the asymptotic expansion satisfies $\|R_j\|_{L^\infty}=O(|j|^{-1})$ which
is less strong than in the non-resonant case.
We recover this issue by gaining a bit more decay from the modified Bessel function to ensure the series converge.
In details, we need to show
\begin{equation*}
\begin{split}
\sum_{\{ j\in\Z, |j|\geq J\}} \|R_j\|_{L^\infty}\Big|
\frac1{\pi}\int_{|\theta_1-\theta_2|}^\pi e^{z\cos(s)} \cos(\sqrt{\nu_j} s) ds-&\frac{\sin(\sqrt{\nu_j}\pi)}{\pi}\int_0^\infty e^{-z\cosh s} e^{-s\sqrt{\nu_j}} ds\Big|
\\&\leq C e^{z\cos(\theta_1-\theta_2)}.
\end{split}
\end{equation*}
To this end, recall $\|R_j\|_{L^\infty}=O(|j|^{-1})$, it suffices to show that
\begin{equation}\label{res-R}
\begin{split}
\Big|
&\frac1{\pi}\int_{|\theta_1-\theta_2|}^\pi e^{z\cos(s)} \cos(\sqrt{\nu_j} s) ds-\frac{\sin(\sqrt{\nu_j}\pi)}{\pi}\int_0^\infty e^{-z\cosh s} e^{-s\sqrt{\nu_j}} ds\Big|\\
\leq& C |j|^{-1} e^{z\cos(\theta_1-\theta_2)}.
\end{split}
\end{equation}
By using integration by parts, we obtain
\begin{equation*}
\begin{split}
&\frac1{\pi}\int_{|\theta_1-\theta_2|}^\pi e^{z\cos(s)} \cos(\sqrt{\nu_j} s) ds\\
=&\frac1{\pi \sqrt{\nu_j}} \Big(e^{z\cos(s)} \sin(\sqrt{\nu_j} s)\big|_{s=|\theta_1-\theta_2|}^{s=\pi}+\int_{|\theta_1-\theta_2|}^\pi e^{z\cos(s)} (z\sin s) \sin(\sqrt{\nu_j} s) ds\Big)
\end{split}
\end{equation*}
and
\begin{equation*}
\begin{split}
&-\frac{\sin(\sqrt{\nu_j}\pi)}{\pi}\int_0^\infty e^{-z\cosh s} e^{-s\sqrt{\nu_j}} ds\\
=&\frac{\sin(\sqrt{\nu_j}\pi)}{\pi \sqrt{\nu_j}} \Big(e^{-z\cosh s} e^{-s\sqrt{\nu_j}}\big|_{s=0}^{s=\infty}+\int_{0}^\infty e^{-z\cosh s} (z\sinh s) e^{-\sqrt{\nu_j}s}ds\Big).
\end{split}
\end{equation*}
Note that
\begin{equation*}
\begin{split}
e^{z\cos(s)} \sin(\sqrt{\nu_j} s)\big|_{s=\pi}=\sin(\sqrt{\nu_j}\pi) e^{-z\cosh s} e^{-s\sqrt{\nu_j}}\big|_{s=0}=\sin(\sqrt{\nu_j}\pi) e^{-z},
\end{split}
\end{equation*}
thus
\begin{equation*}
\begin{split}
\Big|
&\frac1{\pi}\int_{|\theta_1-\theta_2|}^\pi e^{z\cos(s)} \cos(\sqrt{\nu_j} s) ds-\frac{\sin(\sqrt{\nu_j}\pi)}{\pi}\int_0^\infty e^{-z\cosh s} e^{-s\sqrt{\nu_j}} ds\Big|\\
\leq& \frac1{\pi \sqrt{\nu_j}}\Big(e^{z\cos(\theta_1-\theta_2)}+\big|\int_{|\theta_1-\theta_2|}^\pi e^{z\cos(s)} (z\sin s) \sin(\sqrt{\nu_j} s) ds\big|+\\&\qquad\qquad\qquad+
\big|\int_{0}^\infty e^{-z\cosh s} (z\sinh s) e^{-\sqrt{\nu_j}s}ds\big|\Big).
\end{split}
\end{equation*}
We observe that
\begin{equation*}
\begin{split}
&\big|\int_{|\theta_1-\theta_2|}^\pi e^{z\cos(s)} (z\sin s) \sin(\sqrt{\nu_j} s) ds\big|\\
\leq& C \int_{|\theta_1-\theta_2|}^\pi e^{z\cos(s)} (z\sin s) ds\leq C e^{z\cos(\theta_1-\theta_2)}
\end{split}
\end{equation*}
and
\begin{equation*}
\begin{split}
\big|\int_{0}^\infty e^{-z\cosh s} (z\sinh s) e^{-\sqrt{\nu_j}s}ds\big|\leq C\int_{0}^\infty e^{-z\cosh s} (z\sinh s) ds\leq Ce^{-z}.
\end{split}
\end{equation*}
Since $\sqrt{\nu_j}\sim |j|$, therefore we obtain \eqref{res-R}. Hence we complete the proof of Proposition \ref{prop:heat}.
\end{proof}

\section{Proof of Theorem \ref{thm:Stri}}

In this section, we aim to prove Theorem
\ref{thm:Stri} by considering the perturbation both from the magnetic potential $\A$ and electrical potential $a$.\vspace{0.2cm}

We first recall the local smoothing for Klein-Gordon associated with $\LL_{{\A},a}$ which can be proved by following
the same argument as in  \cite[Proposition 4.1]{CYZ}. We state the results but omit the proof.

\begin{proposition}[Local smoothing estimate]\label{prop:lse}
Let $a\in W^{1,\infty}(\mathbb{S}^{1},\mathbb{R})$, ${\A}\in
W^{1,\infty}(\mathbb{S}^{1},\mathbb{R}^2)$, and assume \eqref{eq:transversal}, \eqref{equ:condassa}.
Let $L_{{\A},a}$ be the spherical operator in \eqref{equ:laadefadd}, with first eigenvalue $\mu_1({\A},a)$, and denote by $\nu_0:=\sqrt{\mu_1({\A},a)}$.
Let $v$ be a solution to \begin{equation}\label{eq:kg}
\begin{cases}
\partial_t^2v+\LL_{{\A},a}v+v=0,\quad (t,x)\in\R\times\R^2
\\
v(0,x)=v_0(x),
\\
\partial_tv(0,x)=v_1(x).
\end{cases}
\end{equation}
Then,

$\bullet$ $($Low frequency estimate$)$ for $\beta\in [1, 1+\nu_0)$
\begin{equation}\label{l-local-s}
\begin{split}
\|r^{-\beta}\varphi_0(\sqrt{\LL_{{\A},a}}) v(t,x)\|_{L^2_t(\R;L^2(\R^2))}\leq
C\left(\|v_0\|_{L^{2}(\R^2)}+\|v_1\|_{L^{2}(\R^2)}\right),
\end{split}
\end{equation}

$\bullet$ $($High frequency estimate$)$ for $\beta\in (1/2, 1+\nu_0)$
\begin{equation}\label{h-local-s}
\begin{split}
\|r^{-\beta}(1-\varphi_0)(\sqrt{\LL_{{\A},a}}) v(t,x)\|_{L^2_t(\R;L^2(\R^2))}\leq
C\left(\|v_0\|_{H_{{\A},a}^{\beta-\frac12}(\R^2)}+\|v_1\|_{H_{{\A},a}^{\beta-\frac32}(\R^2)}\right).
\end{split}
\end{equation}

\end{proposition}

\begin{remark}\label{rem:loc}
In particular $a\equiv0$, then $\nu_0=\min_{k\in\Z}\{|k-\Phi_{\A}|\}$. If $v_0, v_1\in \text{span}\{ \psi_{k}(\theta) \}_{k=k_0+1}^\infty$,
then one can take $\nu_0=\sqrt{\mu_{k_0}}$. For example, if $\mu_{k_0}\geq1$, then one can replace $\nu_0$ by $1$ in the above statements.
\end{remark}

In the rest of this section, we use Theorem \ref{thm:Stri-0} and Proposition \ref{prop:lse} to prove Theorem \ref{thm:Stri}.
Since the inhomogeneous Strichartz estimates in \eqref{stri} are direct consequence of the homogenous estimates and Christ-Kiselev lemma \cite{CK} as did in Section \ref{sec:inh},
 we only prove the homogeneous Strichartz estimates in \eqref{stri}, that is, $F=0$.
To this end, without loss of generality,  assuming $u_1=0$, we thus consider the free Klein-Gordon equation
\begin{equation*}
\partial_{t}^2u+\LL_{{\A},a} u+u=0, \quad u(0)=u_0,
~\partial_tu(0)=0.
\end{equation*}
We need to establish \eqref{stri} for $(q,p)\in\Lambda_{s,\eta}$ with $0\leq\eta\leq 1$ and $0\leq s<1$.
More precisely, $(q,p)$ is in  the region $ABCD$ of Figure 1, while $s=1$ in the line $CD$.

 Theorem \ref{thm:Stri} in the case $q=+\infty$ (i.e. the line $AD$) immediately follows by Spectral Theory and the Sobolev embedding \cite[Lemma 2.4]{FZZ}. Indeed, one has
\begin{equation*}
\begin{split}
\|u(t,x)\|_{L^\infty(\R;L^{p}(\R^2))}&\lesssim \|\mathcal{L}_{{\A},a}^{\frac s2} u(t,x)\|_{L^\infty(\R;L^{2}(\R^2))}\\&\lesssim \|u_0\|_{H^{s}_{{\A},a}(\R^2)}\end{split}
\end{equation*}
where $s=1-\tfrac2p$ and $2\leq p<+\infty$. By interpolation, we only need to prove \eqref{stri} for $(q,p)\in\Lambda_{s,\eta}$ with $0\leq\eta\leq 1$ and $\frac12\leq s<1$,
that is, the pairs $(q,p)$ are in the region $CDD'C'$. From now, we only focus on $\frac12\leq s<1$.

We split the initial data into two parts,  $u_0=u_{0,l}+u_{0,h}$ where $u_{0,h}=u_0-u_{0,l}$ and 
 \begin{equation}\label{k0}
 u_{0,l}=\sum_{\{1\leq k\leq k_0\}}a_k(r)\psi_k(\theta),
 \end{equation}
 where $k_0=\max\{k\in \N: \mu_k\leq 1\}$.

\vspace{0.2cm}

%
%
%
\begin{center}
 \begin{tikzpicture}[scale=1]
 \draw[->] (0,0) -- (6,0) node[anchor=north] {$\frac1q$};
\draw[->] (0,0) -- (0,6)  node[anchor=east] {$\frac1p$};
\path (2,-1) node(caption){Fig 1. $0\leq\eta\leq1$};  

 \draw (0,0) node[anchor=north] {O}
 (2.5,0) node[anchor=north] {$\frac14$}
 (5,0) node[anchor=north] {$\frac12$}
 (4,0) node[anchor=north] {$\frac{1+\eta}4$}
 (3,0) node[anchor=north] {$\frac{\eta}2$};

\draw  (0, 5) node[anchor=east] {$\frac12$}
(0, 2) node[anchor=east] {$\frac{\eta}{2(2+\eta)}$};

\draw[<-] (1.8,3) -- (2.3,3.6) node[anchor=south]{$~~~~~~~~\qquad \frac2q=(1+\eta)(\frac1{2}-\frac1p)$};

\filldraw[fill=gray!50](0,5)--(4,0)--(3,0)--(0,2);

\draw[<-] (2.6,1.3) -- (3,2) node[anchor=south]{$~~~~~~~~\qquad\qquad \qquad\qquad \frac12=s=(2+\eta)(\frac1{2}-\frac1p)-\frac1q$};

 \draw[thick]
 (0,5) -- (4,0)
(-0.05, 5)-- (0.05, 5)
(-0.05, 2)-- (0.05, 2)
(0,3.2)--(3.8,0.25)
(2.5,-0.05) -- (2.5,0.05)  (5,-0.05) -- (5,0.05);  
\draw[dashed,thick] (0,2) -- (3,0); 

\draw (0.1,5) node[anchor=west] {\small A};
\draw (0,1.6) node[anchor=west] {\small D};
\draw (0,3.2) node[anchor=west] {\small D$'$};
\draw (4,0.15) node[anchor=west] {\small B};
\draw (2.8,0.15) node[anchor=west] {\small C};
\draw (3.8,0.5) node[anchor=west] {\small C$'$};

\draw (0,2) circle (0.05);
\draw (4,0) circle (0.05);
\draw (3,0) circle (0.05);
\draw (3.8,0.25) circle (0.05)[fill=gray!90];
\draw (0,3.2) circle (0.05)[fill=gray!90];

\end{tikzpicture}

\end{center}

%
%
%
%
%
%
Correspondingly, the solution is splitted into two parts, $u=u_l+u_h$, where $u_l$ and $u_h$ satisfy
\begin{equation}\label{KG-l}
\partial_{t}^2u_l+\LL_{{\A},a} u_l+u_l=0, \quad u(0)=u_{0,l},
~\partial_tu(0)=0,
\end{equation}
and
\begin{equation}\label{KG-h}
\partial_{t}^2u_h+\LL_{{\A},a} u_h+u_h=0, \quad u(0)=u_{0,h},
~\partial_tu(0)=0.
\end{equation}

To prove \eqref{stri} with $u_1=0$ and $F=0$, it suffices to prove
\begin{equation}\label{u-h}
\begin{split}
\|u_h(t,x)\|_{L^q(\R;L^p(\R^2))}\lesssim \|u_{0,h}\|_{ H^{s}_{{\A},a}(\R^2)},
\end{split}
\end{equation}
and
\begin{equation}\label{u-l}
\begin{split}
\|u_l(t,x)\|_{L^q(\R;L^p(\R^2))}\lesssim \|u_{0,l}\|_{ H^{s}_{{\A},a}(\R^2)}.
\end{split}
\end{equation} We first prove \eqref{u-h}. By Duhamel's formula,
we have
\begin{align}\label{h-duhamel'}
u_h(t,x)&=\frac{e^{it\sqrt{1+\LL_{{\A},a}}}+e^{-it\sqrt{1+\LL_{{\A},a}}}}2 u_{0,h}\\\nonumber
&=\frac{e^{it\sqrt{1+\LL_{{\A},0}}}+e^{-it\sqrt{1+\LL_{{\A},0}}}}2 u_{0,h}+\int_0^t\frac{\sin{(t-\tau)\sqrt{1+\LL_{{\A},0}}}}
{\sqrt{1+\LL_{{\A},0}}}(V(x)u_h(\tau,x))d\tau,
\end{align}
with $V(x)=\tfrac{a(\hat{x})}{|x|^2}.$
By using \eqref{h-duhamel'} and Theorem \ref{thm:Stri-0}, we have
\begin{equation*}
\begin{split}
&\|u_h(t,x)\|_{L^q(\R;L^p(\R^2))}\\\lesssim& \|u_{0,h}\|_{ H^{s}_{{\A},0}(\R^2)}+\Big\|\int_0^t\frac{\sin{(t-\tau)\sqrt{1+\LL_{{\A},0}}}}
{\sqrt{1+\LL_{{\A},0}}}(V(x)u_h(\tau,x))d\tau\Big\|_{_{L^q(\R;L^p(\R^2))}}.
\end{split}
\end{equation*}
From \cite[Lemma 2.3]{FZZ}, one has
\begin{equation}\label{equ:equisobeqi}
  \|f\|_{\dot{H}^s_{{\A},0}(\R^2)}\simeq \|f\|_{\dot{H}^s_{{\A},a}(\R^2)},
\end{equation}
for all $s\in[-1,1]$.
Thus our main task now is to prove
\begin{equation}\label{est:h-inh}
\begin{split}
\Big\|\int_0^t\frac{\sin{(t-\tau)\sqrt{1+\LL_{{\A},0}}}}
{\sqrt{1+\LL_{{\A},0}}}(V(x)u_h(\tau,x))d\tau\Big\|_{_{L^q(\R;L^p(\R^2))}}\lesssim \|u_{0,h}\|_{ H^{s}_{{\A},a}(\R^2)}.
\end{split}
\end{equation}
To this end, we estimate
\begin{equation}\label{est:h-inh'}
\begin{split}
&\Big\|\int_0^t\frac{\sin{\big((t-\tau)\sqrt{1+\LL_{{\A},0}}}\big)}
{\sqrt{1+\LL_{{\A},0}}}(V(x)u_h(\tau,x))d\tau\Big\|_{_{L^q(\R;L^p(\R^2))}}
\\\lesssim& \Big\|\int_0^t\frac{\sin{(t-\tau)\sqrt{1+\LL_{{\A},0}}}}
{\sqrt{1+\LL_{{\A},0}}}(1-\varphi_0)(\sqrt{\LL_{{\A},0}})(V(x)u_h(\tau,x))d\tau\Big\|_{_{L^q(\R;L^p(\R^2))}}\\
&+\Big\|\int_0^t\frac{\sin{(t-\tau)\sqrt{1+\LL_{{\A},0}}}}
{\sqrt{1+\LL_{{\A},0}}}\varphi_0(\sqrt{\LL_{{\A},0}})(V(x)u_h(\tau,x))d\tau\Big\|_{_{L^q(\R;L^p(\R^2))}}.
\end{split}
\end{equation}
Hence \eqref{est:h-inh} is the consequence of the following lemma.
\begin{lemma}
For $\frac12\leq s<1$ and $(q,p)\in\Lambda_{s,\eta}$, we have
\begin{equation}\label{est:h-h-inh}
\begin{split}
& \Big\|\int_0^t\frac{\sin{(t-\tau)\sqrt{1+\LL_{{\A},0}}}}
{\sqrt{1+\LL_{{\A},0}}}(1-\varphi_0)(\sqrt{\LL_{{\A},0}})(V(x)u_h(\tau,x))d\tau\Big\|_{_{L^q(\R;L^p(\R^2))}}\\
\lesssim& \|u_{0,h}\|_{ H^{s}_{{\A},a}(\R^2)},
\end{split}
\end{equation}
and
\begin{equation}\label{est:h-l-inh}
\begin{split}
&\Big\|\int_0^t\frac{\sin{(t-\tau)\sqrt{1+\LL_{{\A},0}}}}
{\sqrt{1+\LL_{{\A},0}}}\varphi_0(\sqrt{\LL_{{\A},0}})(V(x)u_h(\tau,x))d\tau\Big\|_{_{L^q(\R;L^p(\R^2))}}\\
\lesssim&\|u_{0,h}\|_{ H^{s}_{{\A},a}(\R^2)}.
\end{split}
\end{equation}

\end{lemma}

\begin{proof}
We first prove \eqref{est:h-h-inh}.
Let $\beta=\frac32-s$ with $\frac12\leq s<1$, then $\frac12<\beta\leq 1$.  We  define the operator $T$ by 
\begin{align*}
   Tf=& r^{-\beta}e^{it\sqrt{1+\LL_{{\A},0}}}(1-\varphi_0)(\sqrt{\LL_{{\A},0}})(1+ \LL_{{\A},0})^{\frac12(\frac12-\beta)} f,\quad f\in L^2(\R^2).
\end{align*}
By \eqref {h-local-s} and Remark \ref{rem:loc}, if $f\in\text{span}\{ \psi_{k}(\theta) \}_{k=k_0+1}^\infty$,  it follows that $T$ is a bounded operator from $L^2(\R^2)$ to $L^2(\R;L^2(\R^2))$, since $\beta\in(\frac12,1]\subset (\frac12,1+\sqrt{\mu_{k_0}})$ (due to the definition of $k_0$ in \eqref{k0}).
By duality, its adjoint $T^*$
\begin{equation*}
T^* F=\int_{\tau\in\R}(1+\LL_{{\A},0})^{\frac12(\frac12-\beta)} (1-\varphi_0)(\sqrt{\LL_{{\A},0}}) e^{-i\tau\sqrt{1+\LL_{{\A},0}}}  r^{-\beta}  F(\tau)d\tau
\end{equation*}
is bounded from $L^2(\R;L^2(\R^2))$ to $L^2(\R^2)$. Define the operator
\begin{equation*}
\begin{split}
B:\,& L^2(\R;L^2(\R^2))\to L^q(\R;L^p(\R^2)), \\ B F&=\int_{\tau\in\R} \frac{e^{i(t-\tau)\sqrt{1+\LL_{{\A},0}}}}{\sqrt{1+\LL_{{\A},0}}}(1-\varphi_0)(\sqrt{\LL_{{\A},0}}) r^{-\beta}F(\tau)d\tau.
\end{split}
\end{equation*}
Hence by the Strichartz estimate \eqref{stri}   with $a\equiv0$ and $s=\frac32-\beta$, one has
\begin{equation}\label{BF}
\begin{split}
&\|B F\|_{L^q(\R;L^p(\R^2))}\\=&\big\| e^{i t\sqrt{1+\LL_{{\A},0}}}\int_{\tau\in\R}\frac{e^{-i\tau\sqrt{1+\LL_{{\A},0}}}}{\sqrt{1+\LL_{{\A},0}}} (1-\varphi_0)(\sqrt{\LL_{{\A},0}})r^{-\beta} F(\tau)d\tau\big\|_{L^q(\R;L^p(\R^2))}\\
\lesssim& \big\|(1+\LL_{{\A},0})^{\frac12(\frac32-\beta)} \int_{\tau\in\R}\frac{e^{-i\tau\sqrt{1+\LL_{{\A},0}}}}{\sqrt{1+\LL_{{\A},0}}}(1-\varphi_0)(\sqrt{\LL_{{\A},0}}) r^{-\beta} F(\tau)d\tau\big\|_{L^2(\R^2)}\\=&\|T^*F\|_{L^2}\lesssim \|F\|_{L^2(\R;L^2(\R^2))}.
\end{split}
\end{equation}
Now we are ready to prove inquality \eqref{est:h-h-inh}. As
$$\sin\big((t-\tau)\sqrt{1+\LL_{{\A},0}}\big)=\frac{1}{2i}\big(e^{i(t-\tau)\sqrt{1+\LL_{{\A},0}}}-e^{-i(t-\tau)\sqrt{1+\LL_{{\A},0}}}\big),$$
 by \eqref{BF}, we have
\begin{equation*}
\begin{split}
&\Big\|\int_\R\frac{\sin{(t-\tau)\sqrt{1+\LL_{{\A},0}}}}
{\sqrt{1+\LL_{{\A},0}}}(1-\varphi_0)(\sqrt{\LL_{{\A},0}}) (V(x)u_h(\tau,x))d\tau\Big\|_{L^q(\R;L^p(\R^2))}\\
\lesssim& \|B(r^{\beta}V(x)u_h(\tau,x))\|_{L^q(\R;L^p(\R^2))}\lesssim \|r^{\beta-2}u_h(\tau,x))\|_{L^2(\R;L^2(\R^2))}
\\
\lesssim& \|u_{0,h}\|_{ H^{\frac32-\beta}_{{\A},a}(\R^2)},
\end{split}
\end{equation*}
where we have used  \eqref{l-local-s} and \eqref{h-local-s} since $\beta\in(\frac12,1]$ satisfies $1\leq 2-\beta<1+\sqrt{\mu_{k_0}}$.
Since $q>2$ and  $s=\frac32-\beta$, by the Christ-Kiselev lemma \cite{CK}, we have \eqref{est:h-h-inh}.  \vspace{0.2cm}

We next prove \eqref{est:h-l-inh} in a similar way. We define the operator
$$T: L^2(\R^2)\to L^2(\R;L^2(\R^2)), \quad Tf= r^{-1}e^{it\sqrt{1+\LL_{{\A},0}}}\varphi_0(\sqrt{\LL_{{\A},0}}) f.$$
From the proof of Proposition \ref{prop:lse} again with $\beta=1$,  it follows that $T$ is a bounded operator.
By duality, its adjoint $T^*$
$$T^*: L^2(\R;L^2(\R^2))\to L^2, \quad T^* F=\int_{\tau\in\R}\varphi_0(\sqrt{\LL_{{\A},0}}) e^{-i\tau\sqrt{1+\LL_{{\A},0}}}  r^{-1}  F(\tau)d\tau$$
is also bounded. Define the operator
$$B: L^2(\R;L^2(\R^2))\to L^q(\R;L^p(\R^2)), \quad B F=\int_{\tau\in\R} \frac{e^{i(t-\tau)\sqrt{1+\LL_{{\A},0}}}}{\sqrt{1+\LL_{{\A},0}}}\varphi_0(\sqrt{\LL_{{\A},0}}) r^{-1}F(\tau)d\tau.$$
Hence again by the Strichartz estimate \eqref{stri} with $a=0$ and $s=(2+\eta)\Big(\frac12-\frac1p\Big)-\frac1q$, one has
\begin{equation}\label{BF'}
\begin{split}
&\|B F\|_{L^q(\R;L^p(\R^2))}\\=&\big\| e^{i t\sqrt{1+\LL_{{\A},0}}}\int_{\tau\in\R}\frac{e^{-i\tau\sqrt{1+\LL_{{\A},0}}}}{\sqrt{1+\LL_{{\A},0}}} \varphi_0(\sqrt{\LL_{{\A},0}})r^{-1} F(\tau)d\tau\big\|_{L^q(\R;L^p(\R^2))}\\
\lesssim& \big\|(1+\LL_{{\A},0})^{\frac s2} \int_{\tau\in\R}\frac{e^{-i\tau\sqrt{1+\LL_{{\A},0}}}}{\sqrt{1+\LL_{{\A},0}}}\varphi_0(\sqrt{\LL_{{\A},0}}) r^{-1} F(\tau)d\tau\big\|_{L^2(\R^2)}\\\lesssim&\|T^*F\|_{L^2}\lesssim \|F\|_{L^2(\R;L^2(\R^2))}.
\end{split}
\end{equation}
Now we estimate \eqref{est:h-l-inh}.
By using \eqref{BF'}, and similar argument as above, for $1/2\leq s< 1$, we have
\begin{equation*}
\begin{split}
&\Big\|\int_\R\frac{\sin{(t-\tau)\sqrt{1+\LL_{{\A},0}}}}
{\sqrt{1+\LL_{{\A},0}}}\varphi_0(\sqrt{\LL_{{\A},0}}) (V(x)u_h(\tau,x))d\tau\Big\|_{L^q(\R;L^p(\R^2))}\\
\lesssim& \|B(r V(x)u_h(\tau,x))\|_{L^q(\R;L^p(\R^2))}\lesssim \|r^{-1}u_h(\tau,x))\|_{L^2(\R;L^2(\R^2))}
\\
\lesssim& \|u_{0,h}\|_{ H^{\frac12}_{{\A},a}(\R^2)} \lesssim \|u_{0,h}\|_{ H^{s}_{{\A},a}(\R^2)}.
\end{split}
\end{equation*}
Due to $q>2$, by the Christ-Kiselev lemma \cite{CK}, we obtain \eqref{est:h-l-inh}.
\end{proof}

Next we aim to prove \eqref{u-l}. Notice that the above argument breaks down since a tighter restriction on $\beta$.
Nevertheless, if $1\leq k\leq k_0$, we can follow the argument of \cite{PST} which treated the radial case. Since
\begin{equation*}
\begin{split}
u_l(t,x)&=\frac{e^{it\sqrt{1+\LL_{{\A},a}}}+e^{-it\sqrt{1+\LL_{{\A},a}}}}2 u_{0,l},\end{split}
\end{equation*}
we only consider the Strichartz estimate for $e^{it\sqrt{1+\LL_{\A,a}}}u_{0,l}$.
By using \eqref{funct}, we write
 \begin{equation*}
\begin{split} e^{it\sqrt{1+\LL_{{\A},a}}}u_{0,l}&=
\sum_{1\leq k\leq k_0}\psi_k(\theta)\int_0^\infty J_{\mu_k}(r\rho)e^{it\sqrt{1+\rho^2}}
\mathcal{H}_{\mu_k}(a_k)\rho d\rho,\\
&=\sum_{1\leq k\leq k_0}\psi_k(\theta)\mathcal{H}_{\mu_k}[e^{it\sqrt{1+\rho^2}}\mathcal{H}_{\mu_k}(a_k)](r),
\end{split}
\end{equation*}
where $\mathcal{H}_{\mu_k}$ is the Hankel transform defined in \eqref{hankel}.

By triangle inequality and the  eigenfunctions $\|\psi_{k}\|_{L^\infty(\mathbb{S}^1)}\leq C$, one has
 \begin{equation}\label{est:l-hom'}
\begin{split} &\|e^{it\sqrt{1+\LL_{{\A},a}}}u_{0,l}\|_{L^q(\R;L^p(\R^2))} \\\leq& C
\sum_{1\leq k\leq k_0}\left\|\mathcal{H}_{\mu_k}[e^{it\sqrt{1+\rho^2}}\mathcal{H}_{\mu_k}(a_k)](r)\right\|_{L^q(\R;L^p_{rdr})}.
\end{split}
\end{equation}
For our purpose, we need the properties on the Hankel transforms which are proved in \cite[Corollary 3.2, Theorem 3.8]{PST}.

\begin{lemma}\label{K0} Let $\mathcal{H}_\mu$ be the Hankel transform of order $\mu$ as defined in \eqref{hankel}, and let $\mathcal{K}^0_{\mu,\nu}:=\mathcal{H}_\mu\mathcal{H}_\nu$. Then

$(\mathrm i)$ $\mathcal{H}_\mu\mathcal{H}_\mu=Id$,

$(\mathrm {ii})$  the operator $\mathcal{K}^0_{\mu,0}$
is bounded on $L^p_{r dr}([0,\infty))$ provided $1<p<\infty$,

$(\mathrm {iii})$  the operator $\mathcal{K}^0_{0,\mu}$ is continuous on $H^s$ provided
$$ -\min\{0,\mu,-s\}< 1<2+\min\{0,\mu,\mu+s\}.$$

\end{lemma}
By using this lemma, the operator $\mathcal{K}^0_{\mu_k,0}$ is bounded in
$L^p_{rdr}([0,\infty))$. Therefore from \eqref{est:l-hom'}, we obtain
 \begin{equation}\label{est:l-hom''}
\begin{split}
&\|e^{it\sqrt{1+\LL_{{\A},0}}}u_{0,l}\|_{L^q(\R;L^p(\R^2))}\\&\leq C \sum_{1\leq k\leq k_0}\left\|(\mathcal{H}_{\mu_k}\mathcal{H}_0)\mathcal{H}_0[e^{it\sqrt{1+\rho^2}}\mathcal{H}_0(\mathcal{H}_0\mathcal{H}_{\mu_k})
(a_{k})](r)\right\|_{L^q(\R;L^p_{rdr})}\\&\leq C \sum_{1\leq k\leq k_0}\left\| \mathcal{H}_0[e^{it\sqrt{1+\rho^2}}\mathcal{H}_0 \mathcal{K}^0_{0,\mu_k}(a_k)](r)\right\|_{L^q(\R;L^p_{rdr})}.\end{split}
\end{equation}
On the other hand, the propagator
\begin{equation*}
\mathcal{H}_0 e^{it\sqrt{1+\rho^2}}\mathcal{H}_0
\end{equation*}
 is the same as the classical Klein-Gordon propagator in the 2D  radial case in which the Strichartz estimates hold.
By using \eqref{est:l-hom''} and (iii) in Lemma \ref{K0} with $\tfrac12\leq s<1$, we obtain
 \begin{equation*}
\begin{split}
\|e^{it\sqrt{1+\LL_{{\A},a}}}u_{0,l}\|_{L^q(\R;L^p(\R^2))}
\leq C\sum_{1\leq k\leq k_0}\left\|\mathcal{K}^0_{0,\mu_k}(a_k)](r)\right\|_{H^s(\R^2)}
\leq C_{k_0}\left\|a_k(r)\right\|_{H^s(\R^2)}.\end{split}
\end{equation*}

In sum, we have proved Theorem \ref{thm:Stri} when $0\leq s<1$. As mentioned in the introduction, to prove
\eqref{stri} in the blanket region $CDO$ of Figure 1, that is, $(q,p)\in\Lambda_{s,\eta}$ with $0\leq\eta\leq 1$ and $s\in[1,\frac{2+\eta}2)$,
we use the Sobolev inequality established in Proposition \ref{prop:sob}. More precisely,
for any $(q,p)\in\Lambda_{s,\eta}$ with $0\leq\eta\leq 1$ and $s\in[1,\frac{2+\eta}2)$, there exists $p_0$ and $s_0\in[0,1)$ such that
$(q,p_0)\in\Lambda_{s_0,\eta}$. Then by Proposition \ref{prop:sob} and the Strichartz estimates proved in above, we obtain
\begin{equation*}
\begin{split}
\|u(t,x)\|_{L^q(\R;L^{p}(\R^2))}&\lesssim \|\mathcal{L}_{{\A},a}^{\frac \sigma 2} e^{it\sqrt{1+\LL_{{\A},a}}}u_{0}\|_{L^q(\R;L^{p_0}(\R^2))}\\&\lesssim
\|u_0\|_{H^{s_0+\sigma}_{{\A},a}(\R^2)}\leq \|u_0\|_{H^{s}_{{\A},a}(\R^2)},\end{split}
\end{equation*}
where $\sigma=2(\frac1{p_0}-\frac1p)\geq0$ and
\begin{equation*}
s=(2+\eta)(\frac12-\frac1p)-\frac1q=s_0+(2+\eta)(\frac1{p_0}-\frac1p)\geq s_0+\sigma.
\end{equation*}

\begin{center}

\end{center}

\end{document}